\documentclass[a4paper]{article}

\usepackage[english]{babel}
\usepackage[utf8]{inputenc}
\usepackage[T1]{fontenc}
\usepackage[noend]{algpseudocode}
\usepackage[algo2e,ruled,linesnumbered,boxed,boxruled,vlined]{algorithm2e}

\usepackage{amsthm,amsfonts,amsmath,amssymb}
\usepackage{fullpage}
\usepackage[square,sort,numbers]{natbib}

\usepackage[shortlabels]{enumitem}
\usepackage{graphicx}
\usepackage{todonotes}
\usepackage{hyperref}
\usepackage{listings}
\usepackage{maplestd2e}

\newcommand{\RR}{\mathbb{R}}
\newcommand{\CC}{\mathbb{C}}
\newcommand{\ZZ}{\mathbb{Z}}

\newcommand{\QQ}{\mathbb{Q}}
\newcommand{\KK}{\mathbb{K}}
\newcommand{\NN}{\mathbb{N}}

\newcommand{\VV}{\mathbb{V}}
\newcommand{\VVK}{\VV_{\KK}}
\newcommand{\VVKT}{\VV_{\KK^{*}}}

\newcommand{\VVCT}{\VV_{\CC^{*}}}

\newcommand{\NP}{\mathsf{NP}\xspace}

\newcommand{\OB}{\mathcal{O}_B}
\newcommand{\OO}{\mathcal{O}}
\newcommand{\sOB}{\widetilde{\mathcal{O}}_B}
\newcommand{\sOO}{\widetilde{\mathcal{O}}}

\newcommand{\bsz}{\mathcal{L}\xspace}
\newcommand{\res}{\mathtt{res}\xspace}
\newcommand{\lc}{\mathtt{lc}\xspace}

\def\norm#1{\| #1 \|}
\def\normi#1{\norm{#1}_{\infty}}

\newcommand{\cD}{\mathcal{D}\xspace}

\newcommand{\cC}{\mathcal{C}\xspace}
\newcommand{\cJ}{\mathcal{J}\xspace}
\newcommand{\cK}{\mathcal{K}\xspace}

\newcommand{\cO}{\mathcal{O}\xspace}

\newcommand{\maple}{\textsc{maple}\xspace}

\newcommand{\Jac}{\ensuremath{\mathtt{Jac}}\xspace}
\newcommand{\msresultant}{\textsc{ms\_resultant}\xspace}
\renewcommand{\emptyset}{\varnothing\xspace}

\newtheorem{theorem}{Theorem}[section]
\newtheorem{prop}[theorem]{Proposition}
\newtheorem{lemma}[theorem]{Lemma}
\newtheorem{cor}[theorem]{Corollary}

\newtheorem{example}[theorem]{Example}
\newtheorem{remark}[theorem]{Remark}
\newtheorem{definition}[theorem]{Definition}

\newtheorem*{theorem-non}{Theorem}

\usepackage{geometry}						%					
\usepackage{graphicx}						%
\usepackage{tikz}							%
\usepackage{enumitem}						%	
\usepackage{graphicx}						%
\DeclareMathOperator{\Vol}{Vol}				%
\newcommand{\Trg}{\ensuremath{\mathsf{TM}_{\Gamma}(\ZZ^2)}\xspace}%
\newcommand{\ous}{\overline{U}}		%
\DeclareMathOperator{\supp}{supp}			%
\newcommand{\TT}{(\KK^*)^2}					%
\DeclareMathOperator{\SL}{SL}				%

\DeclareMathOperator{\Tr}{\top}
\DeclareMathOperator{\Sin}{Sing}				%
\newcommand{\Mfg}{\mathcal{M}_f(\Gamma)}	%

\DeclareMathOperator{\Crit}{Crit}	
\newcommand{\cU}{\mathcal{U}}
\newcommand{\cV}{\mathcal{V}}
\newcommand{\cCr}{\mathcal{C}_{\rho}}

\newcommand{\tmult}{$\mathtt{T}$-multiplicity\xspace}
\newcommand{\TFmult}{{topological face multiplicity set}\xspace}
\newcommand{\TF}{\ensuremath{\mathtt{TF}}\xspace}
\newcommand{\TFK}{\ensuremath{\mathtt{TF}_{\KK}}\xspace}
\newcommand{\TFKg}{\ensuremath{\mathtt{TF}_{\KK}(\Gamma)}\xspace}
\newcommand{\TFKgr}{\ensuremath{\mathtt{TF}_{\KK}^{\rho}(\Gamma)}\xspace}

\newcommand{\TFRg}{\ensuremath{\mathtt{TF}_{\RR}(\Gamma)}\xspace}
\newcommand{\TFRgr}{\ensuremath{\mathtt{TF}_{\RR}^{\rho}(\Gamma)}\xspace}

\newcommand{\TFCg}{\ensuremath{\mathtt{TF}_{\CC}(\Gamma)}\xspace}
\newcommand{\TFCgr}{\ensuremath{\mathtt{TF}_{\CC}^{\rho}(\Gamma)}\xspace}

\newcommand{\amult}{$\mathtt{A}$-multiplicity\xspace}
\newcommand{\AFmult}{{algebraic face multiplicity set}\xspace}

\newcommand{\AFK}{\ensuremath{\mathtt{AF}_{\KK}}\xspace}
\newcommand{\AFKg}{\ensuremath{\mathtt{AF}_{\KK}(\Gamma)}\xspace}
\newcommand{\AFKgr}{\ensuremath{\mathtt{AF}_{\KK}^{\rho}(\Gamma)}\xspace}

\newcommand{\AFRg}{\ensuremath{\mathtt{AF}_{\RR}(\Gamma)}\xspace}
\newcommand{\AFRgr}{\ensuremath{\mathtt{AF}_{\RR}^{\rho}(\Gamma)}\xspace}

\newcommand{\AFCg}{\ensuremath{\mathtt{AF}_{\CC}(\Gamma)}\xspace}
\newcommand{\AFCgr}{\ensuremath{\mathtt{AF}_{\CC}^{\rho}(\Gamma)}\xspace}

\newcommand{\SolKg}{\ensuremath{\mathtt{Sol}_{\KK}(\Gamma)}\xspace}
\newcommand{\SolCg}{\ensuremath{\mathtt{Sol}_{\CC}(\Gamma)}\xspace}

\newcommand{\sjalg}{\textsc{SparseJelonek-2}\xspace}

\colorlet{ecol}{black!50!white}
\colorlet{acol}{blue}
\colorlet{mcol}{red}
\tikzstyle{vertex}=[circle, draw, fill=black, inner sep=0pt, minimum size=2.6pt]
\tikzstyle{mvertex}=[vertex,fill=mcol]
\tikzstyle{edge}=[line width=1.5pt,ecol]
\tikzstyle{anchoredge}=[edge,acol]
\tikzstyle{labelsty}=[font=\normalsize]
\title{Computing the non-properness set of real polynomial maps in the plane
}
\author{Boulos EL HILANY%
\thanks{Institut f\"ur Analysis und Algebra, 
TU Braunschweig, Germany
(\texttt{email:}{boulos.hilani@gmail.com, b.el-hilany@tu-braunschweig.de})
}
\and 
Elias TSIGARIDAS%
\thanks{Inria Paris and Institut de
  Math{\'e}matiques de Jussieu-Paris Rive Gauche, Sorbonne
  Universit\'e and Paris Universit\'e,
  France (\texttt{email:}{elias.tsigaridas@inria.fr})}}

\date{ }

\begin{document}
 \maketitle

\begin{abstract}
  We introduce novel mathematical and computational tools to develop a complete
   algorithm for computing the set of non-properness of polynomials
  maps in the plane. In particular, this set, which we call \emph{the Jelonek
    set}, is a subset of $\KK^2$ where a dominant polynomial map
  $f: \KK^2 \to \KK^2$ is not proper; $\KK$ could be either $\CC$ or $\RR$.
  Unlike all the previously known approaches we make no assumptions on $f$
  whenever $\KK = \RR$; this is the first algorithmm with this property.

  The algorithm takes into account the Newton polytopes of the polynomials. As a byproduct we provide a finer representation
  of the set of non-properness as a union of semi-algebraic curves, that
  correspond to edges of the Newton polytopes, which is of independent interest.
  Finally, we present a precise Boolean complexity analysis of the algorithm and
  a prototype implementation in \maple.
 \end{abstract}
 \newpage
 \tableofcontents
 \newpage

\section{Introduction}

Let $f=(f_1,\ldots,f_n):\KK^n\to \KK^n$ be a polynomial map, where
$\KK \in \{\CC,\RR\}$. We say that $f$ is \emph{non-proper} at a point
$y\in\KK^n$ if for any neighborhood $\cO$ of $y$, the preimage
$f^{-1}(\overline{\cO})$ is not compact, where $\overline{\cO}$ is the Euclidean
closure of $\cO$. Namely, at the set $\cJ_f \subset \KK^{n}$ of points $y$ at
which $f$ is non-proper, there exists a sequence
$\{x_k\}_{k\in\NN}\subset \KK^n$ such that $ \Vert x_k\Vert\rightarrow\infty$
and $f(x_k)\rightarrow y$. We call $\cJ_f$ the \emph{Jelonek set} of $f$.

Jelonek first studied \cite{Jel93} the non-properness of maps $\CC^n\to\CC^n$
for the purpose of pushing forward the state-of-the-art  results around the Jacobian
conjecture~\cite{dEss12}. In this context, if the Jacobian matrix of such a map
$f$ is everywhere non-singular, then the invertibility of $f$ becomes equivalent
to emptiness of its Jelonek set. To this end, substantial work has been directed
towards the study of this
set~\cite{Jel02,JelLas18,Jel99}, which led to solving many problems regarding
the topology of polynomial maps~\cite{Jel99,Fer03,JK03,JelTib17}.

Besides the Jacobian conjecture, the description of the Jelonek set is essential
to compute the atypical values of a polynomial function
$\CC^n\to\CC$~\cite{JelTib17,JK03}, for classifying the algebraic subsets of
$\RR^n$ that are images of $\RR^n$ under a polynomial morphism~\cite{Fer03}, and
for characterizing the set of fixed points under a polynomial
morphism~\cite{JelLas09}. As for applications to real-life problems, polynomial
maps appear as models in algebraic statistics~\cite{LS05},
robotics~\cite{Schicho-rigid-move-20}, computer vision~\cite{DKLT-PLMP-19}, and
chemical reaction networks~\cite{Dic16}; to mention few of them. For such a
setting, the input of the problem is a point in the target space, while the
output is a point in the preimage. Then, the Jelonek set represents some of
those inputs that result in a sub-optimal output.

There is no shortage of effective algorithms for computing exactly the Jelonek
set of some polynomial maps, even in more general
settings~\cite{Jel93,Jel99,Sta02,Sta07}. Unfortunately, they are not universal;
for example, in some cases~\cite{Jel93,Jel01c,Sta02} one requires $\KK$ to be an
algebraically closed field, while in other cases, these methods require the maps
to be finite~\cite{Sta07}. In addition, to our knowledge, there are no precise
bit complexity estimates for the various algorithms. Even more, the known
algorithms rely on black box elimination techniques based on Gr\"obner basis
computations and they do not take into account the structure and the Newton
polytopes of the input. When $\KK = \RR$, the situation is more dire; even
though the geometry of the Jelonek set is fairly well understood
\cite{Jel93,Jel02,Sta07,JelLas18}, to our knowledge, there are no dedicated
algorithms to compute it. Up until now, we had to rely on algorithms that assume
$\KK = \CC$ and compute a superset of the (real) Jelonek set.

\subsection{Our contribution}
\label{sec:intr:main}

We consider a \emph{dominant} polynomial map $f=(f_1,f_2):\KK^2\to \KK^2$, where $\KK \in \{\CC,\RR\}$, i.e., $\overline{f(\KK^2)} = \KK^2$. We present a complete and  efficient algorithm to compute the Jelonek set of $f$, \sjalg (Alg.~\ref{alg:sparse-Jelonek-2}),
along with its mathematical foundations, complexity analysis, and a prototype implementation.

The algorithm makes no assumptions on the input polynomials $f$. It outputs a
partition of the (equations of the) Jelonek set to subsets of semi-algebraic (or
algebraic) curves of smaller degree; hence it provides a more accurate and
detailed picture of the topology and the geometry of the Jelonek set, from the
one known before. It depends on their Newton polytopes; the latter encode the non-zero terms
of the polynomials, see Sec.~\ref{ssubs:mixed} for the definition and various
properties. This feature provides us with tools from combinatorial and toric
geometry, and makes the algorithm input and output sensitive. We present a
precise bit complexity analysis of the algorithm when the input consists of
polynomials with rational coefficients, and a (prototype) implementation in
\textsc{maple}. To our knowledge this is the first dedicated algorithm for
computing the Jelonek set when $\KK = \RR$ in the plane, under no assumptions.

An important aspect of our approach is the partition of the Jelonek set into
sets of irreducible (semi-)algebraic curves.
The equations of the curves in each partition set 
are obtained by analyzing the
coefficients of $f$ restricted at some distinguished edge $\Gamma$ of the
polygon $\NP(f)$; the latter is the Minkowski sum of the Newton polytopes of
$f_1$ and $f_2$, $\NP(f_1)$ and $\NP(f_2)$, respectively. As distinguished edges, 
we consider the ones that their corresponding inner normal vector has at least
one negative coordinate; we call them \emph{infinity edges}
(Definition~\ref{def:various-faces}). 

Consider $q$ to be a generic point of the Jelonek set of $f$, i.e.,
$q \in \cJ_f$, and $y$ sufficiently close to $q$. By definition of $\cJ_f$,
there is an isolated point $x \in f^{-1}(y)$ whose Euclidean norm takes
arbitrary large values. One thus can construct a change of variables that 
sends $x = (x_1, x_2)$ to $(z^{M_1},z^{M_2}) = (z_1^{M_{11}} z_2^{M_{12}}, z_1^{M_{21}} z_2^{M_{22}}) $,
where $M_{ij} \in \ZZ$, so that, in the new variables, the (transformed)
preimage is arbitrary close to one of the coordinate axes of $\KK^2$. This suitable
change of variables depends on an edge of the Newton polytope of $f$ in the
following way: the inverse $U$ of the matrix $M:=(M_1,M_2)\in\ZZ^{2\times 2}$ is
such that its first row spans an edge of $\NP(f)$ (Section~\ref{subsec:base-change}). 
Accordingly, after the change of variables, the new polynomial system $f(x)-y = 0$ 
is written as $\ous (f-y)(z)=0$.

The \emph{\TFmult} corresponding to the face $\Gamma$, or \tmult set
for short, collects all points $q$, for which any $y$ sufficiently close to $q$, gives 
rise to a solution $z$ to the system $\ous (f-y)=0$ satisfying $\Vert z - (\rho,0)\Vert\rightarrow 0$ for some $\rho\in\KK^*$. 

We denote it by $\TFKg$ (Definition~\ref{def:t-face-mult}). 
Going
over all the infinity edges of its Newton polytope $\NP(f)$, we obtain the \tmult set of $f$, and through it
the Jelonek set of $f$. The following theorem summarizes these properties; it
appears in Thm.~\ref{thm:Jf=union-Tmult} along with its proof.
\begin{theorem-non}[The Jelonek set as a union of \tmult sets]
  \label{th:main1}
  Let $f:\KK^2\to\KK^2$ be a dominant polynomial map, where
  $\KK \in \{\CC ,\RR \}$. Then, the Jelonek set of $f$ is the union, over all
  the edges of its Newton polytope $\NP(f)$ intersecting $(\RR^*)^2$, of its \tmult sets.
\end{theorem-non}

To compute the \tmult sets we note that, as $\Vert y - q\Vert \rightarrow 0$, if a
solution $z(y)\in\KK^2$ to any (parametrized by $y$) polynomial system $F_y = 0$
converges to a point $z(q)$, then the multiplicity of $z(q)$ (as a solution of
$F_y = 0$) should usually be higher for $y=q$ than for most other values of $y$.

This observation demonstrates that we can detect points in \tmult sets by
capturing the change in the multiplicity of some solutions of the transformed
system. The set of all such $q\in\KK^2$ that increase the multiplicity of a
solution to $F_y =0$, is called the \emph{\AFmult} corresponding to $\Gamma$, or
\amult for short, and we denote it by $\AFKg$
(Definition~\ref{def:a-face-mult}). We develop an algorithm, \msresultant, to
compute the \tmult set using the \amult set by employing tools from elimination
theory. The following theorem supports \msresultant; it appears in
Prop.~\ref{prop:correct-resultant} along with its proof.
\begin{theorem-non}[\msresultant and computation of the  \tmult sets]
  \label{th:main2}
  Let $f:\KK^2\to\KK^2$ be a dominant polynomial map, where $\KK\in \{\CC, \RR\}$. Then, for any edge $\Gamma$ of $\NP(f)$,
  \msresultant (Alg.~\ref{alg:cim-resultant}),
  correctly computes the \tmult set.
\end{theorem-non}

%\textsc{ms\_Fulton} is based on Fulton's approach for computing solution multiplicities and outperforms
%\msresultant for maps that are  mildly degenerate; alas its has exponential worst case complexity.
\msresultant
relies on resultant computations and exploits the property
that the multiplicity of a root of the resultant accumulates the multiplicities of the solutions of the system in the corresponding fiber.
It is the best choice for maps that are quite degenerate and it has  polynomial worst case complexity.
\msresultant is straightforward for complex polynomial maps. For polynomial maps $f:\RR^2\to\RR^2$,
however, we need to perform a more refined analysis. Namely, let $g:=\CC f:\CC^2\to\CC^2$ denote the complexification of $f$, and let $\mathcal{D}_g$ be its \emph{discriminant}. 
This is the image, under $g$, of the critical points in $\CC^2$ of $g$. For each edge $\Gamma\prec\NP(g)$ we consider the decomposition of $\AFCg\cap\RR^2$ into a collection $\cC$ of smooth disjoint segments $C_1,\ldots,C_r$, separated by isolated points from $\cD_g\cap\AFCg$ and singular points of $\AFCg$. That is, these points are boundary points of the curve segments and $\cJ_f$ is the Euclidean closure of the disjoint union $C_1\sqcup\cdots\sqcup C_r$. 
We show in~\S\ref{sec:f-mult} that each segment $C\in\cC$ either lies entirely in the \tmult set of $f$, or it is disjoint from it. Then, once the \amult set of the complex map $g$ is computed, determining which of its segments $C\in\cC$ is a part of the \tmult set of $f$ reduces to picking a random point in $C$, and counting the number of its real preimages under $f$. One has to make sure that this sampling procedure avoids other \amult sets. We describe this in~\S\ref{sec:mult-set-computation}. Note that, thanks to our first theorem above, this approach holds true if we replace the \amult set by $\cJ_g$ and the \tmult set by $\cJ_f$. 

The complexity bound of \msresultant for maps $\RR^2\to\RR^2$
(Thm.~\ref{thm:cim-res-complexity-R}) is much higher than the one in the complex
case (Thm.~\ref{thm:cim-res-complexity-R}). However, this worst case bound is
attained only for very particular polynomial maps.

Using the previous tools, a straightforward method to compute the
Jelonek set is as follows: For all edges $\Gamma$ of $\NP(f)$ compute the
corresponding \tmult set. This is the backbone of the proof of correctness of algorithm
\sjalg.
\begin{theorem-non}[Computation of the Jelonek set]
  \label{thm:Alg-correctness}
  Let $f:\KK^2\to\KK^2$ be a dominant polynomial map, where $\KK= \{\CC, \RR \}$.
  Then,  \sjalg (Alg.~\ref{alg:sparse-Jelonek-2}) computes correctly
  the Jelonek set of $f$.
\end{theorem-non}

Unlike previously known methods for computing $\cJ_f$, \sjalg depends on the Newton polytope of $f$. Consequently, for maps that are mildly degenerate with respect to their 
Newton polytopes, our methods outperforms standard algorithms for complex maps (cf. Thm.~\ref{thm:Jelonek_2-complexity} and
Cor.~\ref{cor:sparse-complexity}.)

As a byproduct of our methods for describing the Jelonek set, we obtain in Thm.~\ref{thm:special_case-odd} of~\S\ref{sec:f-mult} a partial characterization of real polynomial maps whose Jelonek set is a union of algebraic curves in $\RR^2$. Consequently (Corollary~\ref{cor:birational}), we show that the Jelonek set of real maps, with a birational complexification, is algebraic.

\subsection{An example}
The following example illustrates our first main Theorem that relates the Jelonek set to \tmult sets. Let $f :\KK^2 \to \KK^2$ be the following polynomial map
\begin{equation}\label{eq:example-1}
  f(x) = 0
  \Leftrightarrow \left\{
    \begin{aligned}
      1+2x_1 x_2 - x_1^2x_2^3 & =0 \\
      5 + 12x_1x_2 -10x_1^2 x_2^3+ 2x_1^3 x_2^5  & =0
    \end{aligned}
  \right \}.
\end{equation}
The corresponding Newton polytopes and their Minkowski sum are in
Figure~\ref{fig:ex-mult}.
The preimage of $y$ is such that
$x_1\,x_2 =\frac{2(y_1-1)^2}{(-2y_1 + y_2 - 3)}$. Thus, the norm of $x$ takes
arbitrary large values if $y$ converges to, say, $(-2,-1)$.

To transform the polynomials in~\eqref{eq:example-1} we consider the edge $\Gamma$ (in the Minkowski sum)
with outer normal $(2,-1)$ and defined by the vertices $(2,2)$ and $(5,8)$;
see the rightmost polygon in Figure~\ref{fig:ex-mult}. The corresponding toric
transformation has matrix
$U= \big(\begin{smallmatrix} -1 & 1\\ -2 & 1 \end{smallmatrix}\big)$ and it
transforms the system $f_1(x_1, x_2) - y_1 = f_2(x_1, x_2) - y_2 = 0$ (up to a
monomial multiplication) to
\begin{equation}
  \label{eq:sys:example1}
\left \{
    \begin{aligned}
      2-z_1 - z_2(1-y_1) & =0 \\
      12 -10z_1 + 2z_1^2 -z_2(5-y_2) & =0
    \end{aligned}
  \right \}.
\end{equation}
The new  system has two
simple solutions $(2,0)$ and
$\sigma := \left( \frac{6y_1 - y_2-1}{2(y_1-1)}, \frac{y_2 -2y_1 -3}{2(y_1-1)^2} \right)$.
Notice that if $(y_1, y_2) \rightarrow (-2, -1)$,
then the second coordinate of $\sigma$ goes to 0.
In particular, if  $(y_1, y_2) \rightarrow (-2, -1)$, then  $\sigma \rightarrow (2, 0)$.
For
our example it holds $\TFKg =\{(y_{1}, y_{2}) \in \KK^{2} \,|\, 2y_1 - y_2 - 3=0\}$; which is a curve.
In addition, following Eq.~\eqref{eq:sys:example1},  the solution $(2,0)$ is simple if
$y_2 - 2y_1 -3 \neq 0$. Indeed, the determinant of the Jacobian matrix of the
system~\eqref{eq:sys:example1} evaluated at $z=(2,0)$ equals
$ \big|\begin{smallmatrix} -1 & -2\\ 1-y_1 & 5-y_2 \end{smallmatrix}\big|$. This
is the polynomial the zero set of which defines  \TFKg.

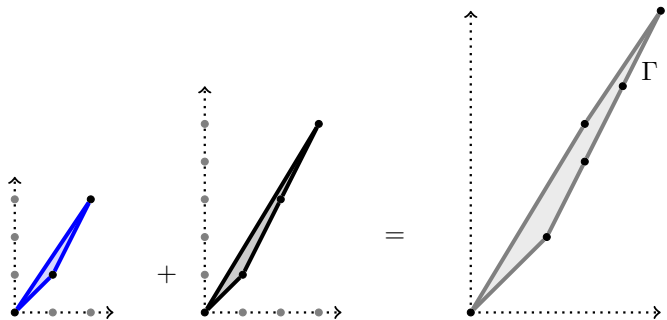
\begin{figure}[t]
  \centering
  \begin{tikzpicture}
    \tikzstyle{bluefill1} = [fill=blue!20,fill opacity=0.8]          % style of one filling
    \tikzstyle{blackfill1} = [fill=gray,fill opacity=0.4]          % style of one filling
    \tikzstyle{greyfill1} = [fill=gray!20,fill opacity=0.8]          % style of one filling

    \begin{scope}[]
			% Axes

	\draw[arrows=->,line width=0.3 mm, dotted] (0,0)-- (0,1.8);

	\draw[arrows=->,line width=0.3 mm, dotted] (0,0)-- (1.3,0);

	\node[vertex, color = gray] (a) at (0.5, 0) {};
	\node[vertex, color = gray] (b) at (1, 0) {};
	\node[vertex, color = gray] (c) at (0, 0.5) {};
	\node[vertex, color = gray] (c) at (0, 1) {};
	\node[vertex, color = gray] (c) at (0, 1.5) {};

			\filldraw[bluefill1,line width=0.0 mm ] (0., 0.)--(0.5, 0.5)--(1, 1.5);	% Drawing of a cylce

			\node[vertex] (1) at (0., 0.) {};
			\node[vertex] (2) at (0.5, 0.5) {};
			\node[vertex] (3) at (1, 1.5) {};
			\draw[edge, color = blue] (1)edge(2) (2)edge(3) (3)edge(1);

		\end{scope}
		\node at (2,.5) {$+$};
		\begin{scope}[xshift=2.5cm]

			% Axes

	\draw[arrows=->,line width=0.3 mm, dotted] (0,0)-- (0,3);

	\draw[arrows=->,line width=0.3 mm, dotted] (0,0)-- (1.8,0);

	\node[vertex, color = gray] (a) at (0.5, 0) {};
	\node[vertex, color = gray] (b) at (1, 0) {};
	\node[vertex, color = gray] (c) at (0, 0.5) {};
	\node[vertex, color = gray] (c) at (0, 1) {};
	\node[vertex, color = gray] (c) at (0, 1.5) {};
	\node[vertex, color = gray] (c) at (0, 2) {};
	\node[vertex, color = gray] (c) at (0, 2.5) {};
	\node[vertex, color = gray] (c) at (1.5, 0) {};

			\filldraw[blackfill1,line width=0.0 mm ] (0., 0.)--(0.5, .5)--(1, 1.5) -- (1.5, 2.5);	% Drawing of a cylce

			\node[vertex] (1) at (0., 0.) {};
			\node[vertex] (2) at (0.5, 0.5) {};
			\node[vertex] (3) at (1, 1.5) {};
			\node[vertex] (4) at (1.5, 2.5) {};
			\draw[edge, color = black ] (1)edge(2) (2)edge(3) (3)edge(4)  (4)edge(1);

		\end{scope}
		\node at (5,1) {$=$};
		\begin{scope}[xshift=6cm]
	\draw[arrows=->,line width=0.3 mm, dotted] (0,0)-- (0,4);

	\draw[arrows=->,line width=0.3 mm, dotted] (0,0)-- (2.5,0);
			\filldraw[greyfill1,line width=0.0 mm ] (0., 0.)--(1, 1)--(2.5, 4) -- (1.5, 2.5);	% Drawing of a cylce

			\node[vertex] (1) at (0., 0.) {};
			\node[vertex] (2) at (1, 1) {};
			\node[vertex] (3) at (2.5, 4) {};
			\node[vertex] (4) at (1.5, 2.5) {};
			\node[vertex] (5) at (1.5, 2) {};
			\node[vertex] (6) at (2, 3) {};
			\draw[edge] (1)edge(2) (2)edge(5) (5)edge(6)  (6)edge(3) (3)edge(4) (4)edge(1);

			\node[label={[labelsty]right:{ $\Gamma$}}] at (2, 3.2) {};

		\end{scope}
      \end{tikzpicture}
      \caption{Two Newton polytopes and their Minkowski sum. They correspond to
        the polynomials of the map in Eq.~\eqref{eq:sys:example1}.}
      \label{fig:ex-mult}
\end{figure}

\subsection{Related work}
There are already works that exploit the  structure of Newton polyhedra for the computation of topological data of polynomial maps.
For polynomial functions this includes computing the Milnor number at the origin~\cite{Kou76} and the bifurcation set~\cite{NZ90,Zah96}.
Whereas for maps, Newton polyhedra were used to compute the \L{}ojasiewicz exponents~\cite{Biv07}, prove non-properness conditions~\cite{thao2009condition}, and compute the set of atypical values~\cite{chen2014invertible,Est13}.
In all these cases, there is the requirement of  some form of face non-degeneracy condition on the corresponding maps. That is, the tuple of polynomials
is required to be in a Zariski open set in the space of all polynomial maps with a given set of Newton polyhedra. Our approach does not need any such assumption whatsoever.

\subsection{Notation}
\label{sec:notation}

We denote by $\OO$, resp. $\OB$, the arithmetic, resp.  bit,
complexity and we use $\sOO$, respectively $\sOB$, to ignore
(poly-)logarithmic factors.

For a polynomial $f \in \ZZ[x]$ or $f \in \ZZ[x_1, x_2]$ its infinity norm
$\normi{f}$ equals the maximum of absolute values of its
coefficients.  We denote by $\bsz(f)$ the logarithm of its infinity
norm.  We also call the latter the bitsize of the polynomial,
that is a shortcut for  the maximum bitsize of all its coefficients.
A univariate polynomial is of size $(d,\tau)$ when its degree is at
most $d$ and has bitsize $\tau$.
We represent a real algebraic number $\alpha \in \RR$ using the
\textit{isolating interval representation}; it includes a square-free
polynomial, $A$, which vanishes at $\alpha$ and an interval with rational
endpoints that contains $\alpha$ and no other root of $A$.
If $\alpha \in \CC$, then instead of an interval we use a rectangle in $\RR^2$
where the coordinates of its vertices are rational numbers.

For a polynomial $f \in \KK[x_1, x_2]$, where $\KK \in \{\RR, \CC\}$, we
denote its zero set by $\VVK(f) \subset \KK^2$ and $\VVKT(f) \subset \TT$ is its
zero set over the corresponding torus. We use the same notation if $f$ is a
polynomial system, that is $f = (f_1, f_2)$.

We use the abbreviation $[n]$ for $\{1, 2, \dots, n \}$.

\subsection{Organization of  the paper}
Sec.~\ref{sec:Jel0} gives the state of the art of the problem. We present the known methods for computing the Jelonek set for maps $\CC^n\to\CC^n$ (Algs.~\ref{alg:Jelonek_n}, and~\ref{alg:Jelonek_2}), and deduce their complexity in Thms.~\ref{Jelonek-n-complexity}, and~\ref{thm:Jelonek_2-complexity}.
In Sec.~\ref{sec:prel} we give the necessary notations to introduce the
algorithm \sjalg. This includes a classification of faces
of $\NP(f)$ and the introduction of the toric change of coordinates.
The first half of Section~\ref{sec:main-alg} gives a detailed description of the
functionality of \sjalg, while in the second half we
present its complexity analysis (Thm.~\ref{th:sparse-complexity}) and a detailed example.
Sections~\ref{sec:f-mult} and~\ref{sec:mult-set-computation} concern the
correctness of Algorithm \sjalg; in particular we define
\tmult sets in terms of appropriate toric transformations introduced in
Sec.~\ref{subsec:base-change}. Then, we reformulate Thm.~\ref{th:main1} as
Thm.~\ref{thm:Jf=union-Tmult} and prove it.
The proof of Thm.~\ref{th:main2} is the proof of Prop.~\ref{prop:correct-resultant} in Sec.~\ref{sec:mult-set-computation}.
Finally, Sec.~\ref{sec:implementation} presents our prototype implementation.

\section{Complexity of known methods}
\label{sec:Jel0}

We start with a definition of the Jelonek set, see~\cite{Jel93,Jel99}.

\begin{definition}[Jelonek set]
  \label{def:Jel}
  Given two affine varieties, $X$ and $Y$, and a map $F:X\to Y$, we say that $F$
  is \emph{non-proper} at a point $y\in Y$, if there is no neighborhood
  $U\subset Y$ of $y$, such that the preimage $F^{-1}(\overline{U})$ is compact,
  where $\overline{U}$ is the Euclidean closure of $U$. In other words, $F$ is
  non-proper at $y$ if there is a sequence of points $\{x_k\}_{k \in \NN}$ in $X$ such that
  $\Vert x_k\Vert\to + \infty$ and $f(x_k)\to y$. The \emph{Jelonek set} of
  $F$, $\mathcal{J}_F$, consists of all points $y\in Y$ at which $F$ is
  non-proper.
\end{definition}

Jelonek proved \cite{Jel93,Jel02} that
%which we also refer to as  Jelonek set,
for a dominant polynomial map $f=(f_1,\ldots,f_n): \KK^n\to\KK^n$,
$ x:=(x_1,\ldots,x_n)\mapsto f(x)$, the set $\cJ_f$, if it is non-empty, then it is
\emph{$\KK$-uniruled}. That is, for every point $y\in \cJ_f$, there exists a
non-constant polynomial map $\varphi:~\KK\to \cJ_f$ such that $\varphi(0)=y$.
Moreover, if $\KK = \CC$, then $\cJ_f$ is a hypersurface~\cite{Jel93}, while in
the real case the dimension is less than or equal to $n-1$. These two properties
are also valid for maps over algebraically closed fields~\cite{Sta05} and when
the domain, or the codomain, is an affine variety~\cite{JelLas18}. In this
setting, there are methods for testing properness~\cite{Jel99} and computing the
Jelonek set~\cite{Jel93,Sta02}. These is also an upper bound on the degree of
$\cJ_f$ when $\KK = \CC$~\cite{Jel93}, which is
\[
\frac{\prod_{i=1}^n\deg f_i -\mu(f)}{\min_{i=1,\ldots,n}\deg f_i} ,
\] where $\mu(f)$ is the number of points in a generic fiber of $f$ (cf.~\cite{Jel93}).

Consider the maps $F_i:~\CC^n\to\CC^n\times\CC$, such that
$F_i(x):=(f(x),x_i)$, for $i \in  [n]$. Also let
\[ \sum\nolimits_{k=0}^{N_i}A_k^i(f) \, x_i^{N_i -k}~\in \CC[f_1,\ldots,f_n,x_i] , \]
be the polynomial defining the equation of the hypersurface $F_i(\CC^n)\subset \CC^n\times\CC$.
Then,  $A_0^i(f)$ is the \emph{$i$-th non-properness polynomial of $f$}.

\begin{theorem}[{{\cite[Prop. 7]{Jel93}}}]\label{th:Jeln}
  The Jelonek set, $\cJ_f$, of a dominant polynomial map $f:~\CC^n\to\CC^n$,  is the zero locus of the polynomial $\prod_{i=1}^nA_0^i(y)$, where each $A_0^i$ is the $i$-th non-properness polynomial of $f$.
\end{theorem}

For the special case $n = 2$, the
following theorem holds:
\begin{theorem}[{{\cite[Thm. 2.2]{Jel01c}}}]\label{th:Je2}
  Consider a dominant polynomial map $f:\CC^2\to\CC^2$, $(x_1,x_2)\mapsto f(x_1,x_2)$.
  Let $P_i(y_1,y_2,x_i)=\sum_{k=0}^{n_i}P_{ik}(y_1,y_2)x_i^{n_i-k}$ be the resultant of the polynomials $(f_1 - y_1, f_2 - y_2)$ with respect to $x_j$ for distinct $i,j\in\{1,2\}$.
  Then, the Jelonek set of $f$ is $\{(y_1,y_2)\in\CC^2~|~ P_{1,0}P_{2,0} = 0\}$.
\end{theorem}

The computation of the implicit equation of parametrized hypersurfaces
$F_i(\CC^n)$ requires to eliminate $n-1$ from the variables $x_1,\ldots,x_n$. Thus, to
compute the Jelonek set we need effective computations with resultants or Gr\"obner bases.

In Alg.~\ref{alg:Jelonek_2} and Alg.~\ref{alg:Jelonek_n} we present the
pseudo-code of the algorithms supported by Thms.~\ref{th:Jeln}
and~\ref{th:Je2} for computing the Jelonek set in $\CC^n$ and $\CC^2$. The
proofs of the following two results are in the appendix.

\begin{theorem}
  \label{Jelonek-n-complexity}
  Let $f = (f_1,\ldots,f_n) \in \ZZ[x_1, \dots, x_n]$ be polynomials of
  size $(d, \tau)$.
  Alg.~\ref{alg:Jelonek_n} computes the Jelonek set $\cJ_f$ in
  \begin{equation}\label{eq:complexity-n}
    \sOB( 2^n n^{n(\omega+1) - \omega +1} d^{n^2 + (n-1)\omega}(\tau + n d)),
  \end{equation}
  where $\omega$ is the exponent in the complexity of matrix multiplication.
  It consists of polynomials in $\ZZ[y_1, \dots, y_n]$
  of size $( \OO(d^n), \sOO((nd)^{n-1} \tau) )$.
\end{theorem}

For the special case of two variables a slightly better bound is possible.
\begin{theorem}
  \label{thm:Jelonek_2-complexity}
  Let $f_1, f_2 \in \ZZ[x_1, x_2]$ be  polynomials of size $(d, \tau)$.
  Alg.~\ref{alg:Jelonek_2} computes the Jelonek set $\cJ_f$ in $\sOB(d^6\tau)$.
  It consists of a polynomial in $\ZZ[y_1, y_2]$ of size $( \OO(d), \sOO(d \tau))$.
\end{theorem}

\begin{algorithm2e}[t]
  %\scriptsize \dontprintsemicolon \linesnumbered
  \SetFuncSty{textsc} \SetKw{RET}{{\sc return}} \SetKw{OUT}{{\sc output \ }} \SetKwInOut{Input}{Input}
  \SetKwInOut{Output}{Output}
    \SetKwInOut{Require}{Require}
    % \SetVline
    \Input{$f = (f_1,\ldots, f_n) \in \ZZ[x_1,\ldots, x_n]^n$}
    %%
    %% \Require{none}
    %%
    \Output{The set of non-properness of $f : \CC^n \to \CC^n$}

    \BlankLine

    \tcc{Notice that $F_i \in (\CC[y])[x]$, where $y = (y_1, \dots, y_n)$.}
    $F_{1} \gets f_{1}(x) - y_{1}, \dots, F_{n} \gets f_{n}(x) - y_{n}$\;

    \For{$i \in [n]$}
    {
      \tcc{Eliminate all the $x$-variables but $x_{i}$.}
      $R_{i} = \res_{x_i}(F_{1}, \dots, F_{n}) = \sum_{k=0}^{N_i}A_k^i(y) x_i^{N_i -k} \in (\CC[y])[x_{i}]$
    }
    \BlankLine
    \RET  $\prod_{i=1}^{n} A_{0}^{i}(y) \in \CC[y]$\;
    \caption{\textsc{Jelonek\_n$(f_1, \ldots,f_n)$}}
  \label{alg:Jelonek_n}
\end{algorithm2e}

\begin{algorithm2e}[ht]
  %\scriptsize \dontprintsemicolon \linesnumbered
  \SetFuncSty{textsc} \SetKw{RET}{{\sc return}} \SetKw{OUT}{{\sc output \ }}
	\SetKwInOut{Input}{Input}
	\SetKwInOut{Output}{Output}
    \SetKwInOut{Require}{Require}
    % \SetVline
    \Input{$F = (f_1, f_2) \in \ZZ[x_1, x_2]^2$}
    %%
    % \Require{none}
    %%
    \Output{The set of non-properness of $F : \CC^2 \to \CC^2$}

    \BlankLine

    $g_1 \gets f_1(x_1, x_2) - y_1 $ \;
    $g_2 \gets f_2(x_1, x_2) - y_2$ \;

    $r_1 \gets \res_{x_2}(g_1, g_2) \in (\ZZ[y_1, y_2])[x_1] $\;
    $r_2 \gets \res_{x_1}(g_1, g_2) \in (\ZZ[y_1, y_2])[x_2] $\;

    \tcc{Return the leading coefficient with respect to $x_1$, resp. $x_2$, of $r_1$ and $r_2$}
    $p \gets \lc_{x_1}(r_1) \cdot \lc_{x_2}(r_2)  \in \ZZ[y_1, y_2]$ \;

    \RET $p$ \;
    \caption{\textsc{Jelonek\_2$(f_1, f_2)$}}
  \label{alg:Jelonek_2}
\end{algorithm2e}

\section{Preliminaries}\label{sec:prel}

We present the necessary terminology we need for the following sections.

\subsection{Polytopes, Minkowski sums, and mixed volume}
\label{subs:faces}
A \emph{polytope}, also called \emph{polygon}, $\Delta$ in $\RR^2$ is a bounded intersection of closed
half-spaces of the form $\{a_0\leq a_1X_1 + a_2X_2\}\subset\RR^2$, where
$a_0,a_1,a_2\in\RR$.
The latter are the  \emph{supporting half-spaces} of
$\Delta$ and  their boundary intersects the boundary of $\Delta$, $\partial\Delta$, at a connected set of
$\Delta$ that we call \emph{face}.
Thus, any face $F$ of $\Delta$  minimizes a
function
$a^*:\Delta\rightarrow\RR$,
given by 
$(X_1,X_2)\mapsto - a_0 + a_1X_1+a_2X_2$.
In this case $a=(a_1,a_2)\in\RR^2$ is an interior normal vector to $F$
and we say that it \emph{supports} $F$.

The \emph{Minkowski sum} $A \oplus B$ of two subsets $A,B\subset\RR^n$ is the
set $\{a+b~|~a\in A,~b\in B\}$. Let $A_1$ and $A_2$ denote the respective convex hulls of two
finite sets in $\ZZ^2$ and let their Minkowski sum be $A = A_1 \oplus A_2$. The
\emph{summands} of $A$ refers to the pair $(A_1,A_2)$. If $\Gamma$ is a face of
$A$, that is $\Gamma \prec A$, then we denote the summands of $\Gamma$ by
$(\Gamma_1, \Gamma_2)$, such that $\Gamma = \Gamma_1 \oplus \Gamma_2$,
$\Gamma_1 \prec A_1$, and $\Gamma_2 \prec A_2$.

\begin{figure}[t]
  \centering
  \begin{tikzpicture}
    \tikzstyle{bluefill1} = [fill=blue!20,fill opacity=0.8]          % style of one filling
    \tikzstyle{blackfill1} = [fill=gray,fill opacity=0.4]          % style of one filling
    \tikzstyle{greyfill1} = [fill=gray!20,fill opacity=0.8]          % style of one filling

	\begin{scope}[scale = 0.9]

    \begin{scope}[]
			% Axes

	\draw[arrows=->,line width=0.3 mm, dotted] (0,0)-- (0,1.3);

	\draw[arrows=->,line width=0.3 mm, dotted] (0,0)-- (1.8,0);

	\node[vertex, color = gray] (a) at (0.5, 0) {};
	\node[vertex, color = gray] (b) at (1, 0) {};
	\node[vertex, color = gray] (c) at (0, 0.5) {};
	\node[vertex, color = gray] (c) at (0, 1) {};
	\node[vertex, color = gray] (c) at (1.5, 0) {};

			\filldraw[bluefill1,line width=0.0 mm ] (0., 0.)--(1., 0.5)--(1.5, 1) -- (1, 1);	% Drawing of a cylce

			\node[vertex,label={[labelsty]below:{\normalsize $a_0$}}] (1) at (0., 0.) {};
			\node[vertex] (2) at (1., 0.5) {};
			\node[vertex] (3) at (1.5, 1) {};
			\node[vertex] (4) at (1, 1) {};
			\draw[edge, color = blue] (1)edge(2) (2)edge(3) (3)edge(4)  (4)edge(1);

			\node[label={[labelsty]right:{ $a_1$}}] at (0.5, 0.25) {};
			\node[label={[labelsty]below:{ $a_2$}}]  at (1.5, 1) {};
			\node[label={[labelsty]right:{ $a_3$}}]  at (1, 1.25) {};
			\node[label={[labelsty]above:{ $a_4$}}]  at (.5, .5) {};

		\end{scope}
		\node at (2,.5) {$ \oplus$};
		\begin{scope}[xshift=2.5cm]

			% Axes

	\draw[arrows=->,line width=0.3 mm, dotted] (0,0)-- (0,2.6);

	\draw[arrows=->,line width=0.3 mm, dotted] (0,0)-- (5,0);

	\node[vertex, color = gray] (a) at (0.5, 0) {};
	\node[vertex, color = gray] (b) at (1, 0) {};
	\node[vertex, color = gray] (c) at (0, 0.5) {};
	\node[vertex, color = gray] (c) at (0, 1) {};
	\node[vertex, color = gray] (c) at (0, 1.5) {};
	\node[vertex, color = gray] (c) at (0, 2) {};
	\node[vertex, color = gray] (c) at (1.5, 0) {};
	\node[vertex, color = gray] (c) at (2, 0) {};
	\node[vertex, color = gray] (c) at (2.5, 0) {};
	\node[vertex, color = gray] (c) at (3, 0) {};
	\node[vertex, color = gray] (c) at (3.5, 0) {};
	\node[vertex, color = gray] (c) at (4, 0) {};
	\node[vertex, color = gray] (c) at (4.5, 0) {};
	\node[vertex, color = gray] (c) at (5, 0) {};

			\filldraw[blackfill1,line width=0.0 mm ] (0., 0.)--(2, .5)--(3.5, 1) -- (4.5, 1.5) -- (5, 2)-- (2, 2) ;	% Drawing of a cylce

			\node[vertex,label={[labelsty]below:$b_0$}] (1) at (0., 0.) {};
			\node[vertex] (2) at (2, .5) {};
			\node[vertex] (3) at (3.5, 1) {};
			\node[vertex] (4) at (4.5, 1.5) {};
			\node[vertex] (5) at (5, 2) {};
			\node[vertex] (6) at (2, 2) {};
			\draw[edge, color = black ] (1)edge(2) (2)edge(3) (3)edge(4)  (4)edge(5)
			 (5)edge(6)  (6)edge(1);

			\node[label={[labelsty]right:{ $b_1$}}] at (0.6, 0) {};
			\node[label={[labelsty]below:{ $b_2$}}] at (2.8, .9) {};
			\node[label={[labelsty]below:{ $b_3$}}] at (4, 1.3) {};
			\node[label={[labelsty]below:{ $b_4$}}] at (5, 2) {};
			\node[label={[labelsty]right:{ $b_5$}}] at (2.8, 2.3) {};
			\node[label={[labelsty]right:{ $b_6$}}] at (0.6, 1.5) {};
		\end{scope}
		\node at (8,1) {$=$};
		\begin{scope}[xshift=8cm]

			\filldraw[greyfill1,line width=0.0 mm ] (0., 0.)--(2, .5)--(3.5, 1) -- (5.5, 2) -- (6.5, 3)-- (3, 3) ;	% Drawing of a cylce

			\node[vertex] (1) at (0., 0.) {};
			\node[vertex] (2) at (2, .5) {};
			\node[vertex] (3) at (3.5, 1) {};
			\node[vertex] (4) at (5.5, 2) {};
			\node[vertex] (5) at (6.5, 3) {};
			\node[vertex] (6) at (3, 3) {};
			\draw[edge] (1)edge(2) (2)edge(3) (3)edge(4)  (4)edge(5)
			 (5)edge(6)  (6)edge(1);

			\node[label={[labelsty]right:{ $a_0\oplus b_1$}}] at (0.6, 0) {};
			\node[label={[labelsty]right:{ $a_0\oplus b_2$}}] at (2.2, 0.5) {};
			\node[label={[labelsty]right:{ $a_1\oplus b_3$}}] at (4, 1.3) {};
			\node[label={[labelsty]right:{ $a_4\oplus b_6$}}] at (0.6, 2.3) {};
			\node[label={[labelsty]right:{ $a_2\oplus b_4$}}] at (6, 2.5) {};
			\node[label={[labelsty]right:{ $a_3\oplus b_5$}}] at (3.4, 3.2) {};

		\end{scope}

	\end{scope}
  \end{tikzpicture}

  \caption{Two Newton polytopes and  their Minkowski sum. They correspond
    to the polynomials of the map in Section~\ref{subsec:example}. The edges of the blue, resp. the black, polytope are labeled by $a_i$, resp.  $b_j$, for  $i,j\neq 0$.
    The vertices that correspond to $(0,0)$ are  $a_0$, and $b_0$.}
  \label{fig:pol}
\end{figure}
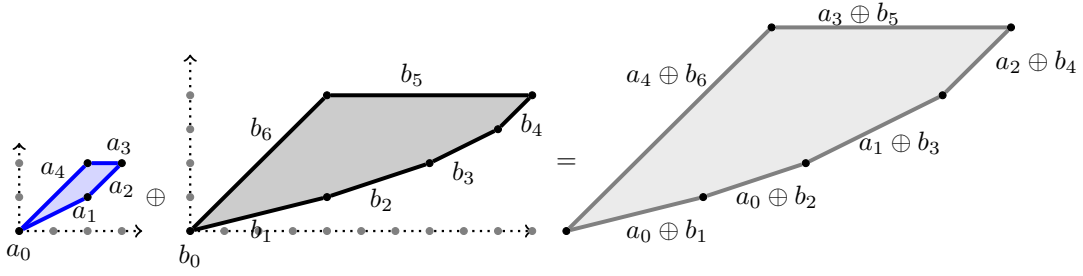

\subsubsection{Mixed volume}\label{ssubs:mixed}
Given a convex set $\Delta\subset\RR^2$, let $\Vol (\Delta)$ be its fixed
and Lebesgue measure endowed in $\RR^2$.
Minkowski's \emph{mixed volume} is the unique real-valued multi-linear, with
respect to the Minkowski sum, function of two convex sets
$\Delta_1,\Delta_2\subset\RR^2$, whose value, if $\Delta_1 = \Delta_2 = \Delta$,
equals $2\Vol(\Delta)$.
We denote the  mixed
volume by $\mathtt{MV}(\Delta_1,\Delta_2)$ and we can compute it  using the
inclusion-exclusion formula
\[
  \Vol(\Delta_1 \oplus \Delta_2) - \Vol(\Delta_1) - \Vol(\Delta_2).
\]

If $\Delta_1 \oplus \Delta_2$ is a line segment or if $\Delta_i$ is
a point for some $i\in \{1,2\}$, then $\mathtt{MV}(\Delta_1,\Delta_2)=0$. The other direction of this statement is
also true; it is a particular case of Minkowski's theorem for the
higher-dimensional mixed volume, see~\cite[Section 2]{Kho16}.

A pair $(\Delta_1,\Delta_2)$ is \emph{independent} if
$\mathtt{MV}(\Delta_1, \Delta_2)\neq 0$; or, equivalently, if
$\dim (\sum_{i\in I} \Delta_i)\geq |I|$ for all $I\subset \{1,2 \}$. A pair is
\emph{dependent} if it  is not independent.

\subsubsection{Characterization of the faces}
\label{subs:long-short}

In what follows, we distinguish types of edges of the Minkowski sum of $\Delta:=\Delta_1 \oplus \Delta_2$.
The following definition
details these types, while Figure~\ref{fig:edge-partition} gives a pictorial overview.

\begin{figure}[ht]
  \centering
  \includegraphics[scale=0.9]{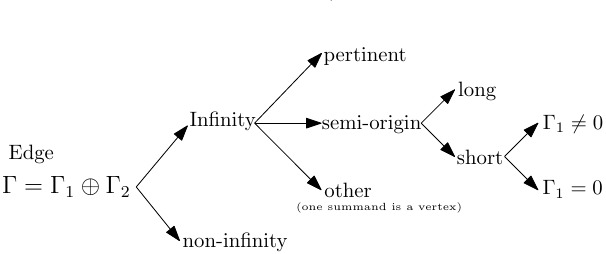}
  \caption{The partition of the edges that \sjalg
    exploits. $\Gamma$ is an edge of the Minkowski sum and its summands are
    $\Gamma_{1}$ and $\Gamma_{2}$. }
  \label{fig:edge-partition}
\end{figure}

\begin{definition}
  \label{def:various-faces}
  Let $\Delta = \Delta_1 \oplus \Delta_2$.
  A face $\Gamma \prec \Delta$ with summands $(\Gamma_1, \Gamma_2)$ is
  an \emph{edge} if $\dim(\Gamma) = 1$. An edge is
  \begin{itemize} %[labelindent=5mm,topsep=0mm,noitemsep]
  \item \emph{long} if both of its summands have dimension 1, that is if
    $\dim(\Gamma_1) = \dim(\Gamma_2) = 1$.
  \item \emph{short} if it is not long.
  \item \emph{pertinent} if it is long and both of its summands do not contain the origin, that is
    $(0, 0) \not\in \Gamma_1$ and $(0, 0) \not\in \Gamma_2$.
  \item \emph{semi-origin} if at least one of its summands contains the origin, that is
  $(0, 0) \in \Gamma_1$ or $(0, 0) \in \Gamma_2$.
  \item \emph{origin} if both of its summands contain the origin, that is
  $(0, 0) \in \Gamma_1$ and $(0, 0) \in \Gamma_2$.
  \item \emph{infinity} if the corresponding (inner) normal vector has a negative coordinate.
  %if all of its supporting vectors have a negative coordinate.
  \end{itemize}
\end{definition}

%\begin{remark}\label{rem:infinity_faces}
%  If  $a:=(a_1,a_2)\in\QQ^2$ supports an infinity semi-origin edge of $A$, then it satisfies $a_1a_2<0$. Hence, a pertinent edge is also infinity.
%\end{remark}

\begin{example}\label{ex:faces}
Consider the three (Newton) polytopes in Figure~\ref{fig:pol}.
The third polytope is the
Minkowski sum of the first two and we  have the following characterization of its edges:
\begin{itemize}
  \item Semi-origin long edges: $a_1\oplus b_3$, $a_4\oplus b_6$.
  \item Origin long edges: $a_4\oplus b_6$.
  \item Semi-origin short edges: $a_0\oplus b_1$, $a_0\oplus b_2$.
  \item Origin short edges: $a_0\oplus  b_0$, $a_0\oplus b_1$.
  \item Pertinent edges: $a_2\oplus b_4$, $a_3\oplus b_5$.
%  \item Almost semi-origin edges: $a_2\oplus b_4$.
  \item All of the edges are infinity edges.
\end{itemize}
\end{example}

\subsection{Polynomials restricted to faces}
\label{ssubs:restr}
Let  $P\in\KK[x_1^{\pm 1},x_2^{\pm 1}]$ be a bivariate Laurent polynomial, that is
\[ P(x) = \sum_{a \in\ZZ^2} c_a \, x^a
  = \sum_{(a_1, a_2) \in\ZZ^2} c_a \, x_1^{a_1} x_2^{a_2} ,\]
where  $c_a\in\KK$.
The \emph{support} of $P$ is the set $\supp(P) = \{a \in\ZZ^2~|~c_a\neq 0\}$.
The Newton polytope of $P$ is the convex hull of its support;
we denote it by  $\NP(P)$.

Consider the pair of polynomials
$f :=(f_1, f_2)\in\KK[x_1^{\pm 1},x_2^{\pm 1}]^2$,
with non-zero constant terms such that
\begin{equation}
  \label{eq:f1-f2}
  f_1 = \sum\nolimits_{a \in S_1} c_{1, a} x^{a}
  \quad \text{ and } \quad
  f_2 = \sum\nolimits_{b\in S_2} c_{2, a} x^{b} ,
\end{equation}
where $S_i = \supp(f_i) \subset \ZZ^2$.
The corresponding Newton polytopes (in our case polygons)
are $\Delta_i = \NP(f_i)$, for $i \in \{1,2\}$.
Also let $\NP(f) := \Delta = \Delta_1 \oplus \Delta_2 $.
For any face $\Gamma \prec \Delta$ with summands $(\Gamma_1, \Gamma_2)$,
%\footnote{Maybe to put $\subset$ instead of $\prec$}
we denote by $f_{\Gamma_1}$ the restriction of $f_1$ to those monomial
terms $c_{1, a} x^{a}$ for which  $a \in \Gamma_1 \cap \ZZ^2$.
Similarly for $f_{\Gamma_2}$.
We also write $f_\Gamma$ for the
pair
$(f_{\Gamma_1}, f_{\Gamma_2})\in\KK[x_1^{\pm 1},x_2^{\pm 1}]^2$.

\section{Decomposing the Jelonek set}
\label{sec:f-mult}

Let $f:\KK^2\to\KK^2$ be a dominant polynomial map sending $(0,0)$ to
$(\KK^*)^2$. The main goal of this section is to describe $\cJ_f$ in
terms of the multiplicities and the existence of solutions of polynomial systems resulting from
suitable toric transformation of $f$.
The results of this section contribute to the
correctness proof of the main algorithm in Sec.~\ref{sec:mult-set-computation}.

We keep using $\Delta_1$, $\Delta_2$, and $\Delta$ to denote $\NP(f_1)$, $\NP(f_2)$, and
$\NP(f_1) \oplus \NP(f_2)$, respectively.

\subsection{Toric change of variables using edges}
\label{subsec:base-change}

To each edge $\Gamma\prec\Delta$, we introduce a toric change of coordinates $U\in\SL(2,\ZZ)$
to deduce a description of the points in the preimage of
$f$ that escape to infinity. 

\begin{figure}[t]
  \centering
  \includegraphics[scale=1.6]{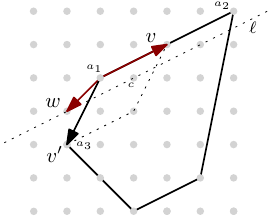}
  \caption{Toric change of basis. The face $\Gamma$ is the edge
      delimited by $a_1$ and $a_2$, while $\Gamma'$ is delimited by $a_1$ and
      $a_3$. The corresponding primitive vectors are $v$ and $v'$, respectively.
      The vectors $v$ and $w$ consist a basis of $\ZZ^2$. In addition they
      define a cone that contains all the lattice points of the polygon. Thus,
      the toric change of basis consists of the  vectors $v$ and $w$ and
      $\tilde{e} = (v, w)$. }
  \label{fig:basis}
\end{figure}

Assume that $(\Gamma_1,\Gamma_2)$ is the pair of summands of $\Gamma$, and
let $a_1$ be any endpoint of $\Gamma$. We use $v$ to denote the primitive vector
starting from $a_1$ along the direction of $\Gamma$. 

Then, there exists a vector $w\in\ZZ^2$ satisfying
  $\det(v,w)=\pm 1$ and such that for any point in $\Delta\oplus \{-a_1\}$
  spanned by the basis $\tilde{e}:=(v,w)$, the coefficient in front of $w$ is
  positive.
  In other words, $w$ points towards the direction of the polygon.
    We call $(v,w)$ a \emph{$\Gamma$-basis for $\Delta$}.

The linear transformation
$U:\ZZ^2\to\ZZ^2$ maps the new basis $\tilde{e}$ to the canonical basis of $\ZZ^2$.
We can also relate the computatin of $\tilde{e}$ to the computation
of the shorted vector problem in $\ZZ^{2}$ \cite[Lecture~VIII]{Yap-algebra-book}.

Now consider the matrix $T$, where
\[
  T =
  \big(\begin{smallmatrix} t_{11} & t_{12}\\ t_{21} & t_{22} \end{smallmatrix}\big)
  =
  \big(\begin{smallmatrix} v_{1} & v_{2}\\ w_{1} & w_{2} \end{smallmatrix}\big)
  \in \SL(2,\ZZ) .
\]
It corresponds to  the following change of variables $z = x^{T \binom{1}{1}}$.
The transformation $U$ that we are looking for is the transpose of the inverse of $T$,
that is $U = T^{-\top}$ which induces an isomorphism $(\KK^*)^2\to (\KK^*)^2$, given by
the transformation
\[
 (z_1,~z_2) \mapsto \big( z_{1}^{w_{2}/D}z_{2}^{-v_{2}/ D},~ z_{1}^{-w_{1}/D}z_{2}^{v_{1}/D}\big),
\]
where $D = \det(T) =  \pm 1$ is the determinant of $T$.
We deduce that $T$,
and thus also $U$, depends on
$\Gamma$ and $\Delta$. 

\begin{definition}\label{def:Trg}
  The set of \emph{$\Gamma$-toric transformations} is the subset of all matrices
  $\big(\begin{smallmatrix} w_{2}/D & -w_{1}/D\\ -v_{2}/D & v_{1}/D \end{smallmatrix}\big) \in\SL(2,\ZZ)$,
 where $(v,w)$ is a $\Gamma$-basis for $\Delta$.
  We denote this set by $\Trg$.
\end{definition}

%$\Trsf_{E\prec A}(\ZZ)$, or simply by $\Trg$ whenever $A$ is fixed.

We also have the following immediate consequence. %(see Example~\ref{ex:transf1}).

\begin{lemma}\label{lem:vector-change}
  Let $\alpha\in\ZZ^2$ be the primitive integer vector supporting an edge $\Gamma \prec \Delta$.
  Then, for any $U\in\Trg$, $U\Gamma$ is an edge of $U\Delta$.  Moreover,
  there is a $U$ such that the vector $U^{-\Tr}\cdot\alpha $ equals $(0,1)$ and it
  supports $U\Gamma$.
\end{lemma}

\begin{remark}
  \label{rem:Ustar-0dim}
  If $f=(f_1,f_2)$ is a pair of bivariate polynomials and $Uf = (Uf_1,Uf_2)$,
  then $\VVKT(f)$ has the same number of isolated points as $\VVKT(Uf)$.
\end{remark}

If $Uf$ consists of Laurent polynomials, then there might be monomials with
negative exponents. We transform them to polynomials by multiplying them by
suitable monomials. We denote this transformation by $\ous f$, which is the map
$\left(U x^{r_1}f_1, U x^{r_2}f_2 \right)$, for suitable $r_1, r_2 \in \ZZ^2$;
in other words we multiply by monomials of the smallest possible degree to clear
the denominators.

The following observations will be useful in the sequel
as they relate the roots of the transformed polynomial system to the
roots of the original one.

 \begin{remark}[Toric transformation and polynomials]
   \label{rem:one-var}

   For any  $f \in \KK[x_1,x_2]^2$, where  $\Delta = \NP(f)$,
   the following hold:
   \begin{enumerate}[(i)]
     \item For any $U\in\Trg$, the system $\ous f_\Gamma =0$ is univariate and $\ous f=0$ is bivariate, and
     \item $\ous f_\Gamma =0$ has a solution $\rho \in\KK^*$ iff
     $\ous f =0$ has a solution $(\rho,0)\in \KK^*\times\{0\}$ iff
     $f_\Gamma = 0$ has a solution in $\TT$.
   \end{enumerate}
 \end{remark}

 \subsection{The topological and algebraic multiplicity sets of $f$}
\label{subs:main}

Let $f:\KK^2\to\KK^2$ be a polynomial map whose Newton polytope is denoted by $\Delta$ and let $U\in\Trg$ for some fixed $\Gamma\prec\Delta$. Define the graph
\[
G:=G(U,f)=\left\lbrace\left. (z,y)\in(\KK^*)^2\times\KK^2~\right|~\ous(f-y)(z)=0 \right\rbrace,
\] 
and consider the projection $\pi:\KK^2\times\KK^2\to\KK^2$, $(z,y)\mapsto y$.

\begin{definition}[\TFKg, \textbf{\TFmult}]
  \label{def:t-face-mult} 

  For any $\rho\in\KK^*$, consider the following set
  \[
  T^{\rho}(\Gamma) := \left\lbrace y\in\KK^2~|~\exists~\{(z_k,y_k)\}_{k\in\NN}\subset G(U, f),~z_k\underset{k\rightarrow\infty}{\longrightarrow}(\rho,0)~\text{ and }~y_k\rightarrow y \right\rbrace .
\]
Then $\TFKgr$ is  the
Euclidean closure $T^{\rho}(\Gamma)$, that is $\TFKgr = \overline{T^{\rho}(\Gamma)}^{E}$.

  The \emph{\TFmult} of $f$ corresponding to $\Gamma$, $\TFKg$,
  is the union of the sets $\TFKgr$, that is
  \[
    \TFKg = \bigcup\nolimits_{\rho} \TFKgr ,
  \]
  where $\rho$ takes all distinct values in $\KK^*$
  satisfying $\ous(f - \tilde{y})((\rho, 0))=0$ for some $\tilde{y}\in\KK^2$.
  For brevity,  we call it \tmult set.
\end{definition}

We prove later in
Thm.~\ref{thm:Jf=union-Tmult} that the Jelonek set of $f$ is formed by the union of all \tmult
sets, running over all the infinity edges of $\NP(f)$.

Depending on the type of the edge $\Gamma\prec\Delta$, a non-empty set $\TFKgr$ can be either a finite set or a semi-algebraic curve.

 We represent those cases in two examples.
\begin{example}\label{ex:top-mult}
Consider the case $g:=\ous(f-y)$, where $g_1= - y_1 + z_1^2 + z_2$ and $g_2 = -y_2+z_1^2$. For any $\rho\in\KK^*$, consider the sequence $z(k):=(\rho,~1/k)$ for $k\in\NN$. If $y(k)=(\rho^2+1/k,~\rho^2)$, then the sequence $\{(z(k),~y(k))\}_{k\in \NN}\subset G$ converges to $(\rho,~0,~\rho^2,~\rho^2)$. Hence, it holds $\TFKgr = \{(\rho^2,~\rho^2)\}$, and thus $\TFKg=\{(\rho^2,~\rho^2)~|~\rho\in\KK\}$.
\end{example}

\begin{example}\label{ex:top-mult2}
Assume now that $g:=\ous(f-y)$ is given by $g_1= 1 - z_1 - y_1z_2$ and $g_2 = 1 + 2z_1 -3z_1^2 -y_2z_2$. On the one hand, for each $i\in\{1,2\}$, it holds $g_i(z_1,0)= 0$ $\Leftrightarrow$ $z_1=1$ making the set $\TFKgr$ empty whenever $\rho \neq 1$. On the other hand, as the sequence $z(k):=(1+1/k,~1/k)$ converges to $(1,~0)$, the sequence $y(k):=(1,~4+3/k)$ satisfies $g(z(k),y(k))=0$, and converges to the line $L:=\{4y_1 = y_2\}$. We deduce that if $\rho = 1$, then $\mathtt{TF}_{\KK}^{\rho}(\Gamma) = \TFKg = L$. 
\end{example}

The following observation follows from the definitions.

\begin{lemma}\label{lem:multip}
  The \tmult set of $f$ corresponding to $\Gamma$, \TFKg, does not depend on the choice of the toric
  transformation $U$ from $\Trg$.
\end{lemma}

To compute \tmult sets we need to introduce the multiplicity of a solution of a polynomial system.

\begin{definition}{\cite[Ch.~4, Def.~2.1]{CLO2}}
    Let $g := (g_{1}, g_{2}) \in \KK[z_{1}, z_{2}]^{2}$ and $J = \langle g_1, g_{2}\rangle$
    the corresponding ideal. The \emph{multiplicity} of a solution $x\in\KK^2$ of the system $g=0$
    is the dimension of the ring obtained by localizing $\KK[z_{1}, z_{2}]$ at the maximal ideal
    $m := \langle z_{1} - x_{1}, z_{2}-x_{2} \rangle$
    and considering the quotient ring $\KK[z_{1}, z_{2}]_{m} / J \, \KK[z_{1}, z_{2}]_{m}$.
\end{definition}

For our purposes, to compute the multiplicity of the solution of a bivariate
polynomial system we proceed as follows. We augment the system with a linear form,
say $g_{0} = s - r_{1} z_{1} - r_{2} z_{2}$, where $s$ is a new variable and $r_{\{1,2\}}$ are suffiently generic integers. After we eliminate the variables $z_{1}, z_{2}$ from the system $g_{0} = g_{1} = g_{2} = 0$, we obtain a univariate polynomial in $s$. The multiplicity
of the roots of this last polynomial correspond to the multiplicities of the roots of the system.
For the further details of this approach we refer the reader to \cite{blmprs-bivsolve-16,BrSa-0dim-16,Rouillier-rur}.

For any point $\rho\in\KK$ and any $y\in \KK^2$, we denote by $\mu_\rho(y)$ the multiplicity of a solution $(\rho,0)$ to the system $\ous (f-y)=0$.

\begin{definition}[\AFKg, \textbf{\AFmult}]
  \label{def:a-face-mult}
  For any $\rho\in\KK^*$, define $\AFKgr$ to be the 
  set of all points $y\in\KK^2$ such that
  for any $\tilde{y} \not \in \AFKgr$, it holds
  \[
  \mu_\rho(y)> \mu_\rho(\tilde{y}).
  \]
  The \emph{\AFmult} of $f$ corresponding to $\Gamma$, \AFKg, is the union
  of all \AFKgr, where $\rho$ takes all distinct values in $\KK^*$ satisfying
  $\ous(f - \tilde{y})((\rho, 0))=0$ for some $\tilde{y}\in\KK^2$.

  For brevity, we also call it \amult set.
\end{definition}

\begin{example}\label{ex:alg-mult}
Consider the polynomials $(g_1,g_2)=(- y_1 + z_1^2 + z_2,~-y_2+z_1^2)$ from Example~\ref{ex:top-mult}. If $\KK = \CC$, then the equalities $\TFCg = \{(\rho^2,~\rho^2)\in\CC^2~|~\rho\in \CC\} = \{y_1 - y_2 = 0\}= \AFCg\subset\CC^2$ follow from the fact that $y_1 = y_2$ if and only if the system $\ous (f-y)=0$ has a solution of the form $(\rho,0)$ for some $\rho\in\CC^*$. Similarly, if $\KK = \RR$, then $\TFRg = \{(\rho^2,~\rho^2)\in\RR^2~|~\rho\in \RR\} = \{y_1 - y_2 = 0\}\cap \RR^2_{\geq 0}= \AFRg$. 
\end{example}

\begin{example}\label{ex:alg-mult2}
Consider the polynomials $(g_1,g_2)=(1 - z_1 - y_1z_2,~-1 + 2z_1 -3z_1^2 -y_2z_2)$ from Example~\ref{ex:top-mult2}. For any $\KK\in\{\RR,~\CC\}$, the point $(1,~0)$ is a simple solution to the system $\ous(f-y)=0$ whenever $y\not\in L$. Indeed, the line $L$ is the zero locus of the polynomial in $\KK[y_1,y_2]$ obtained as the determinant of the Jacobian matrix $\det \Jac_{(1,0)} \ous(f-y)$. Therefore, it holds $L=\AFKg$.
\end{example}

As the previous examples indicate, for any $\Gamma\prec\Delta$, it holds $\TFKg \subset \AFKg$. Indeed, since the system $\ous(f-y)=0$, from Def.~\ref{def:t-face-mult}, has a
  solution $\varrho\in\KK^*\times\{0\}$ whose multiplicity increases at points
  $y\in\TFKg$, we get the inclusion ``$\subseteq$''. We will see (Prp.~\ref{prop:tmult=amult-C}) that we have equality if $\KK = \CC$, and the inclusion may be strict for $\KK = \RR$ 
%  depending on the type of $\Gamma$ 
  (Rem.~\ref{rem:irred_real-semi} and Thm.~\ref{thm:special_case-odd}). Furthermore, notice (see e.g.,~Example~\ref{ex:alg-mult}) that, 
  %if $g:=\CC f$ is the complexification of $f:\RR^2\to\RR^2$, then, 
  in general, we have $\TFRg \neq \AFCg\cap\RR^2$.

In Section~\ref{sec:test-empty-mset}
we present an algorithm to indentify the components (if any) of \AFRg contributing to \TFRg
and hence to $\cJ_f$.

\subsection{The Jelonek set as a union of multiplicity sets}\label{subs:Jel-union}

We start by introducing some useful notation and auxiliary results. We use \SolKg to
denote the set of numbers $\rho \in \KK$ for which $(\rho,0)$ is a solution of
the system $\ous(f-y) = 0$, for some $y\in\KK^2$ and $U\in\Trg$, where $\Gamma$ is an edge of $\Delta$.
We denote by $\mathcal{D}_f$ the \emph{discriminant} of $f$. That is, if
$ C_f:=\left\lbrace x\in\KK^2~|~\det(\Jac_x(f))=0\right\rbrace$, then
$\mathcal{D}_f = f(C_f)$.

Since $f$ is a dominant map, the following holds.

\begin{lemma}\label{lem:discriminant-Axis}
  The set $\mathcal{D}_f\cup f \big(\KK^2\setminus \TT\big)$ is a finite union of curves.
\end{lemma}

We need the following technical lemma for the proof
Thm.~\ref{thm:Jf=union-Tmult}. It identifies the edges that we can safely
exclude from our computations.

\begin{lemma}
  \label{lem:short}
  Let $\Gamma \prec \Delta$ be an edge.
  If one of the summands of $\Gamma$ is a vertex that is not the origin,
  then for any
  $y\in\KK^2$, the system
  $(f-y)_\Gamma = 0$ has no solutions in $(\KK^*)^2$ .

\end{lemma}
\begin{proof}
  Let $\Gamma = \Gamma_1 \oplus \Gamma_2$. Assume that $\Gamma_1$ is the vertex that is not the origin.
  Then,
  $(f_1 - y_1)_{\Gamma} = c_w x^w$ and $c_w \not= 0$.
  Thus, its solution is $0$ and so the system $(f - y)_{\Gamma} = 0$
  has no solutions in the torus $\TT$.
\end{proof}

Now we are ready to state and prove our main result.

\begin{theorem}
  \label{thm:Jf=union-Tmult}

   Consider a dominant polynomial map  $f:\KK^2\rightarrow\KK^2$ and its Newton polytope $\Delta$,
  where $\KK \in \{\CC, \RR\}$.
  The Jelonek set of $f$ is 
  \begin{equation}
    \label{eq:union-mult}
    \cJ_f = \bigcup_{\Gamma}\TFKg,
  \end{equation}
  where $\Gamma$ runs over all infinity edges of $\Delta$.
\end{theorem}
\begin{proof}
  Consider the subset $\cJ_f^\bullet \subseteq  \cJ_f$ containing points $y$ such that
  $f^{-1}(y)$ is finite%
  \footnote{For example consider the blowup $f : (x_1, x_2)\mapsto (x_1, x_1 x_2)$.
    Then,
    $\cJ_f^\bullet = \{(0,y_2)~|~y_2\in\KK^*\}$
    and  $f^{-1}(y) = \varnothing$, for any $y \in \cJ_f\setminus\{(0,0)\}$, and hence it is a finite set.}.
  Thus,
  \[ \cJ_f^\bullet = \{ y \in \cJ_f \,|\, \dim(f^{-1}(y)) \leq 0\} . \]
  On the other hand,
  \[ \cJ_f\setminus \cJ_f^\bullet  = \{ y \in \cJ_f \,|\, \dim(f^{-1}(y)) = 1\}, \]
  and so the preimage of any point in 
  $\cJ_f\setminus \cJ_f^\bullet$ 
  shares with $C_f$ an irreducible curve.
  In addition the  $\cJ_f\setminus \cJ_f^\bullet$ is finite, i.e., it has dimension 0,
  as there are finitely many points $y \in \KK^2$ for which the system
  $(f-y)(x) = 0$ is not zero dimensional.
  So $\cJ_f^\bullet$ is dense in $\cJ_f$.

  Let $T^{\rho}(\Gamma)$ be the subset of points
  $y\in  \TFKg$ as
  in Def.~\ref{def:t-face-mult} (without taking the Euclidean closure).
%  The $\KK$-uniruledness of
%  $\cJ_f$ and $\Mfg$ (see~\cite{Jel02} and Proposition~\ref{prop:char-mult}~\ref{it:2} respectively)
%  implies that
  It is enough to show that
  \begin{equation}
    \label{eq:union-mult:open}
    \cJ_f^\bullet = \bigcup_{\Gamma} T^{\rho}(\Gamma),
  \end{equation}
  where the faces $\Gamma$ are the infinity edges of $\Delta$. This is so, because both
  sets, on the left and on the right hand side of Eq.~\eqref{eq:union-mult:open}, are
  dense subsets of the corresponding closed (under the Euclidean topology) sets.

  To prove the inclusion $\cJ_f^\bullet \subseteq  \bigcup_{\Gamma} T^{\rho}(\Gamma)$,
  we borrow the arguments of the proof of~\cite[Thm. B]{Ber75}   
  that relates the branches of the curves at infinity to the faces of the corresponding Newton polytopes
  (cf. also~\cite[Lemma 1]{Gwozdziewicz2016}).
   Lem.~\ref{lem:discriminant-Axis}, the fact that the Jelonek set is
  $\KK$-uniruled~\cite{Jel93,Jel02} (see also the proof of
  Prop.~\ref{prop:correct-resultant}) imply that the discriminant, the \tmult
  sets

  and the image of the coordinate axes form a union of curves in $\KK^2$.

  Therefore, for any $y\in \cJ_f^\bullet$, one
  can choose a point $q\in\KK^2$, and a line segment $\lambda(]0,1])$ outside
  \[
    %\text{\eqref{eq:union-mult}}
     \bigcup_{\Gamma}\TFKg \cup\mathcal{D}_f \cup \cJ_f^\bullet\cup f\big(\KK^2\setminus\TT\big),
  \]
  where $\lambda:=(\lambda_1,\lambda_2):~[0,1]\to\KK^2$, $t\mapsto (1-t)y + tq$. 
  Then, we consider the parametrized polynomial system
\begin{equation}
  \label{eq:parametr}
  f - \lambda(t) = 0.
\end{equation}
Our assumptions on $\lambda$ imply that all the solutions $x(t)\in\KK^2$
of~\eqref{eq:parametr} are simple and contained in $\TT$.
Moreover, the corrdinates of every solution is a function in $t$.
In particular, the coordinates of the solutions have the following  Puiseux series expansion
\begin{equation}\label{eqref:Puis}
  x_i(t) = a_it^{\alpha_i}+~\text{higher order terms in $t$}, \text{ for } i \in \{1,2\},
\end{equation}
where $a_i\in\KK^*$ and $\alpha_i\in\QQ$. 
% These are solutions of the system~\eqref{eq:parametr}, where we extend the
% field of coefficients to the field of Puiseux series defined over $\KK$.
If we substitute $x(t)$ in  $f_i - \lambda_i(t)$
and set to zero the  coefficient of the smallest power of $t$, say $\omega^*$,
then
we obtain $(f_i - y_i)_{\Gamma}(a_1,a_2)=0$, for some edge $\Gamma\prec\NP(f)$.
Indeed, the value $\omega^*$ is  $\min(\langle\alpha,\omega\rangle~|~\omega\in \NP(f_i))$,
which is attained only for points $w$ in the face of $\NP(f_i)$ supported by the vector $\alpha=(\alpha_1,\alpha_2)$.
Therefore, the point $a=(a_1,a_2)\in\TT$ is a solution to $(f-y)_\Gamma=0$.

To prove that  $\Gamma$ is an infinity edge
we use the definition of the Jelonek set (Def.~\ref{def:Jel}).
As  $y\in \cJ_f^\bullet$,
one of the solutions $x(t)$ of~\eqref{eq:parametr}
converges to infinity as $t\to 0$ (we may take, e.g. a branch from $f^{-1}(\lambda(]0,1[))$). 
%there exists a sequence $\{x_k\}$ of isolated solutions of the system
%$f(x) = 0$ converging to infinity, while $f(x_k) \to y$.
%We can choose such a sequence so that  $f(x_k)\in \lambda([0,1])$, for all $k$.
%The curve
%selection Lemma implies that one of the solutions $x(t)$ %of~\eqref{eq:parametr}
%converges to infinity as $t\to 0$. 
Then, the value $\alpha_i$ appearing
in~\eqref{eqref:Puis} is negative for some $i$. Together with
Lem.~\ref{lem:short}, it implies that $\Gamma$ is an infinity edge.

To prove that $y\in T^{\rho}(\Gamma)$, we combine the previous
description with toric change of coordinates. The vector
$(\beta_1,\beta_2):=\big(U^{-1})^{\Tr}\cdot\alpha$ which supports the edge
$\Xi:=U\Gamma$ of the transformed polytope $U\NP(f)$ equals $(0,1)$
(Lem.~\ref{lem:vector-change}). The point $z(t)$, satisfying $z^U(t)=x(t)$, is
a solution to $ \ous(f -\lambda(t)) = 0$, and we have
\begin{equation}\label{eq:Puis2}
z_i(t)=b_it^{\beta_i}+~\text{higher order terms in $t$}, \text{ for } i \in  \{1,2\},
\end{equation} for some $b=(b_1,b_2)\in\TT$. Therefore, as $t\to 0$, $z(t)$ converges to a point $\varrho\in\KK^*\times\{0\}$.

%This proves the first inclusion.

%, that is a solution to $\ous (f-y)=0$. This shows that the multiplicity of $\varrho$ is higher for $\ous (f-y)=0$, than for $\ous (f-p)=0$.
%Since the number of isolated solutions (counted with multiplicity) is locally preserved, the multiplicity of $\varrho$ is higher for $\ous (f - y)=0$ than for $\ous f=0$.

To prove the  inclusion  $\cJ_f^\bullet \supseteq  \bigcup_{\Gamma} T^{\rho}(\Gamma)$,
let $\Gamma\prec \NP(f)$ be an infinity edge.
%semi-origin or pertinent. 
Assume that
$T^{\rho}(\Gamma)\neq\varnothing$ for some matrix $U\in\Trg$.
For any $y\in T^{\rho}(\Gamma)$, let $q$ be close enough to
$y$, which results a solution $p\in\TT$ to $\ous (f-q)=0$ as in
Def.~\ref{def:t-face-mult}. 
As $q$ converges to $y$, the second coordinate of $p$ converges to $0$, while
the first one remains close to a non-zero constant. Then, we can express $p$ by
a Puiseux series of the form~\eqref{eq:Puis2}, with $\beta =(0,1)$. The matrix $U\in\Trg$ transforms the point $p$ into a solution $p^U$ to the
system $f-q=0$, whose coordinates are expressed as in~\eqref{eqref:Puis}.
The assumptions on $\Gamma$ imply that one of the $\alpha_i$ appearing
in~\eqref{eqref:Puis} is negative. Then, the solution $p^U$ converges to
infinity as $p\to\varrho$, and $q\to y$. This proves that $y\in\cJ^\bullet_f$.
\end{proof}

\subsubsection{The \tmult sets in $\CC$}

In what follows, we describe the relation between \tmult and the \amult sets for different cases of $\KK$. 

\begin{prop}
  \label{prop:tmult=amult-C}
  Consider a dominant polynomial map  $f:\CC^2\rightarrow\CC^2$, its Newton polytope $\Delta$, and let $\rho \in \SolCg$. Then, for any edge $\Gamma\prec\Delta$, it holds
  \begin{equation}\label{eq:prop:tmult=amult-C}
    \TFCgr = \AFCgr .
  \end{equation}  
\end{prop}

{

\begin{proof}

 We first prove the statement for the case where $\Gamma$ is a semi-origin edge. Recall the surface $G\subset (\CC^*)^2\times\CC^2$, and $\overline{G}\subset \CC^4$ defined in Sec.~\ref{subs:main}. Then, Def.~\ref{def:t-face-mult} shows that $\TFCg = \pi(\overline{G}\cap \{z_2=0\})$. Furthermore, since $\Gamma$ is semi-origin, the value $y_i$ appears in the equation $\ous (f_i-y_i)_\Gamma(z_1,z_2) = \ous (f_i-y_i)(z_1,0)$ for some $i\in\{1,2\}$ (see Remark~\ref{rem:one-var}) . Hence, for any $\rho\in\CC^*$ and $y\in \CC^2$, it holds $\ous(f-y)(\rho,0)=0$ if and only if $(\rho,0,y_1,y_2)\in \overline{G}\cap \{z_2=0\}$. We deduce that $y\not\in \pi(\overline{G}\cap \{z_2=0\})$ $\Rightarrow$ $y\not \in  \AFCg$, which yields the equality $\TFCg =\AFCg$.

Note that, for a semi-origin edge $\Gamma$, the above arguments hold true after we replace $\CC$ by $\RR$. Thus, we also get $\TFRg =\AFRg$. The same cannot be said for pertinent edges.

Assume now that $\Gamma$ is pertinent, and let $T(\Gamma)$ be the subset of $\TFCgr$ defined as
  in Def.~\ref{def:t-face-mult} but without taking the Euclidean closure.

  The inclusion $T(\Gamma) \subseteq \AFCgr$ follows
  directly from the definition of the multiplicity sets. Since $\AFCgr$ is
  closed, the closure of $T(\Gamma)$ (i.e. the set
  $\TFCgr$) is included in $\AFCgr$.

  To prove the inclusion $\TF_{\CC}(\Gamma) \supseteq \AFCgr$
  we proceed as follows. Consider the set $A(\Gamma)$
  of points $y\in\AFCgr$ at which the system $\ous (f-y)=0$ has finitely-many
  solutions. One can check (see e.g. proof of Thm.~\ref{thm:Jf=union-Tmult}) that
  $\AFCgr\setminus A(\Gamma)$ is finite and hence
  $A(\Gamma)$ is dense in $\AFCgr$.

  Next, we
  choose a point $q\in A(\Gamma)$, and
    will show that it
  belongs to $\TF_{\CC}(\Gamma)$. Then, the point $(\rho,0)$ is
  an isolated solution to $\overline{U}(f-q)=0$. 

  Let $R\in \CC[y_1, y_2, z_1]$ be the generator in the elimination ideal
\begin{equation}\label{eq:elim}
    \langle\ous(f_1 - y_1),~\ous(f_2 - y_2)\rangle\cap \CC[y_1,y_2,z_1],
\end{equation}  and let $k$ be the largest integer allowing the factorization $R=(z_1-\rho)^kR_0$, with $R_0:=\phi_0(y) + \phi_1(y)(z_1-\rho) +\cdots + \phi_r(y)(z_1-\rho)^r$. 

Note that, we have $R_0(q,\rho)= 0$, and $R_0(\tilde{q},\rho)\neq 0$ for any $\tilde{q}\not\in \AFCgr$.  There exists $m>0$ such that the polynomial $R_0$ is expressed as
  \begin{equation}\label{eq:QR012}
	\phi_0\cdot R_{01} + (z_1-\rho)^mR_{02},  
  \end{equation} where $R_{01}$ collects all terms of $R_0$ that are factors of $\phi_0$, and $R_{02}$ is not a factor of $(z_1-\rho)$. Moreover, we have $R_{02}$ is not identically zero since $\forall y\in\CC^2$, $\exists z\in\CC^2$ such that $\ous(f-y)(z)\neq 0$. This shows that $\AFCgr$ is given by the curve $\VV(Q) $, where $\phi_0$ is a factor of $Q\in\CC[y]$.

Now, for any generic point $\tilde{q}\rightarrow q$, we get $Q(\tilde{q})\rightarrow 0$, and the solutions to $\ous(f - \tilde{q})=0$ are mapped to the roots $\Sigma_{\tilde{q}}\subset\CC$ of $R_0(\tilde{q},z_1)$ under the projection $\pi_1:\CC^2\to\CC$, $(z_1,z_2)\mapsto z_1$. The point $\tilde{q}$, together with each such solution $\sigma\in\Sigma_{\tilde{q}}$ satisfy
\begin{equation}\label{eq:Rational}
\phi_0(\tilde{q}) = -R_{02}(\tilde{q},\sigma)/R_{01}(\tilde{q},\sigma).
\end{equation} It follows from~\eqref{eq:QR012} that the rational function $\varphi_{\tilde{q}}:\CC\to\CC$, defined as $z_1\mapsto-\frac{R_{02}}{R_{01}}(\tilde{q},z_1)$ satisfies $\varphi_{\tilde{q}}(\rho)=0$.  Therefore, whenever $Q\rightarrow 0$, equation~\eqref{eq:Rational} has a solution $\sigma\rightarrow \rho$. We deduce that a point in the set
\[
\pi_1^{-1}(\sigma)\cap\VVCT(\ous(f-q)) 
\] converges to the vertical line $\pi^{-1}(\rho)$, containing $(\rho,0)$. 
To show that at least one of the points in $\Sigma_\sigma$ converges to $(\rho,0)$ it is
enough to note that the same arguments apply if we exchange $z_1$, $\rho$, and the line $\pi_1^{-1}(\rho)$ with $z_1$, $0$, and the line $\pi_2^{-1}(0)$. 

We have shown that $A(\Gamma)\subset\TFCgr$. This
yields the inclusion  "$\supset$", of~\eqref{eq:prop:tmult=amult-C} as
$A(\Gamma)$ is a dense subset of $\AFCgr$ and
$\TFCgr$ is closed.

\end{proof}

\begin{remark}\label{rem:irred_real-semi}
Using the same arguments as in the proof Prop.~\ref{prop:tmult=amult-C}, we can extract two further statements: 
\begin{enumerate}

	\item the set $\AFCg$ is a union of irreducible algebraic curves, and
	
	\item if $f$ is a real map instead, and $\Gamma$ is a semi-origin edge, then $\AFRgr = \TFRgr$. 

\end{enumerate}
\end{remark}

Let us further describe the multiplicity sets for the complex case.

\begin{lemma}\label{lem:weight}
Let $f$ and $\Delta$ be as in Prop.~\ref{prop:tmult=amult-C}, and let $S$ be an irreducible component of $\AFCg$ (see Rem.~\ref{rem:irred_real-semi}). Then, for any $p\not\in S$, and any $q\in S\setminus\Sigma$ for some finite subset $\Sigma\subset S$, the following two quantities are equal and independent of the choice of $p$ and $q$.
\begin{enumerate}
	\item The number of solutions in $(\CC^*)\times\{0\}$ of the system $\ous(f-q)=0$ if $\Gamma$ is semi-origin, and 

%	\item if $\Gamma$ is semi-origin, the system $\ous(f-q)=0$ has exactly $m$ solutions in $(\CC^*)\times\{0\}$, whereas $\ous(f-p)=0$ has none, and 
	
	\item The difference $\mu_\rho(p)-\mu_\rho(q)$ if $\Gamma$ is pertinent.

\end{enumerate}
We denote the value of theses two quantities by $m$, which we also call the \emph{weight} of the algebraic multiplicity.
\end{lemma}

\begin{proof}
Since the map is complex, the set $\overline{G}\cap \{z_2 = 0\}$ is a union of irreducible algebraic curves. Then, Prop.~\ref{prop:tmult=amult-C} shows that $S$ is the image of a component $\cC\subset\overline{G}\cap \{z_2 = 0\}$ under the projection $\pi:(x,y)\mapsto y$. 

Assume first that $\Gamma$ is semi-origin. Then, for some $i\in\{1,2\}$, the equation $\ous(f_i - y_i)(z_1,0)=0$ can be written as $\ous(f_i)(z_1,0) - y_i = 0$. Hence, the map $K:\cC\to \{z_2 = 0\}$, given by the restriction to $\cC$ of the other projection $(z,y)\mapsto z$, is a generically finite map, where $K^{-1}(q)$ is a set of solutions to $\ous(f-q)=0$ in $(\CC^*)\times\{0\}$. The value $m\in\NN$ will thus be exactly the topological degree of $K$.

Assume now that $\Gamma$ is pertinent. Then, the above map $K$ is not dominant, and $S = \AFCgr$ for some $\rho\in\CC^*$. Following the steps of the proof of Prop.~\ref{prop:tmult=amult-C} and using its notations we conclude that $\AFCgr$ is given by a polynomial $Q\in\CC[y_1,y_2]$ that divides $\phi_0$. Since the polynomial~\eqref{eq:QR012} is obtained from the ideal~\eqref{eq:elim}, the value $m$ appearing therein is exactly the difference of multiplicities $\mu_\rho(p)-\mu_\rho(q)$. Here, we set $\cV\subset\VV(Q)\setminus \VV(\phi_{i_1})\cup \cdots\cup \VV(\phi_{i_k})$, where $R_{02} = \phi_{i_1}z^{i_1} + \cdots + \phi_{i_k}z^{i_k}$.
\end{proof}

\subsubsection{The \tmult sets in $\RR$}

Let $f:\KK^2\to\KK^2$ be a dominant polynomial map, let $\Delta$ be its Newton polytope, let $\Gamma$ be an edge of $\Delta$, and let $U\in\Trg$. We use $g$ to denote the pair of polynomials $(g_1,g_2):=\ous(f-y)(z)$. 
 Recall that $G$ is given by $\{(z,y)~|~g=0\}\subset (\KK^*)^2\times\KK^2$, and $\overline{G}$ is given as the closure of $G$ in $\KK^4$. Let $P:=\pi_{|G}$, where $\pi:\KK^4\to\KK^2$, $(z,y)\mapsto y$.

\begin{lemma}\label{lem:jacobian-preserved}
It holds that
\[
\overline{P(\Crit P)} = \overline{\mathcal{D}_f\cap\TT} .
\]
\end{lemma}
\begin{proof}
First, note that $P(\Crit P)$ can be expressed as 
\[
P\left(\left\lbrace (z,y)\in G~|~\det \Jac_z g =0\right\rbrace\right)
\] From its construction, the map $g$ is written as
\[
(z,y)\longmapsto (z^{s_1}(f_1\circ \varphi - y_1),~z^{s_2}(f_2\circ \varphi - y_2)),
\] for some $s_1,s_2\in\NN^2$, where $\varphi:\TT_z\to\TT_x$ is the coordinate change of variables induced by  $U $. 

Let $(z,y)\in G$. Then, there exists $v\in\ZZ^2$ such that
\[
\Jac_z g = z^{v}\cdot\Jac_z(f_1\circ \varphi - y_1,~f_2\circ \varphi - y_2),
\] and since $z_1\cdot z_2\neq 0$, we get 
\[
\det \Jac_z g =0 \Leftrightarrow \det\Jac_z(f\circ\varphi).
\] The multivariate chain rule implies that 
\[
\Jac_z f\circ \varphi = (\Jac_{\varphi} f\circ \varphi)\cdot \Jac_z \varphi,
\] where $\det \Jac_z \varphi$ is non-zero for any $z\in \TT$. Therefore, we have $(z,y)\in G$ and $ \det\Jac_z(f\circ\varphi) =0$ if and only if there exists $x\in \TT$ such that $x=\varphi (z)$, $f(x)=y$ and $ \det\Jac_x(f) = 0$. This readily yields the proof.
\end{proof}

In what follows, we assume that $\KK=\RR$. Let $S\subset\RR^2$ be a semi-algebraic curve, such that there exists a polynomial surjective map $\phi:\RR\to S$. Then $S$ will be called a \emph{parametric semi-line}. Jelonek showed in~\cite{Jel02} that $\cJ_f$ can be decomposed into a finite union of parametric semi-lines. In the rest of this section, we use Theorem~\ref{thm:Jf=union-Tmult} to obtain several results on this decomposition.

We say that $p\in S$ is an \emph{endpoint} if for any small enough $\varepsilon >0$, the open ball $B_\varepsilon(p)\subset \RR^2$, intersected with $S$ is homeomorphic (in the Euclidean topology) to the half-open interval $[0,1[$. 
%Let $\CC f:\CC^2\to\CC^2$ denote the extended map, given by the complexification of $f$.

\begin{theorem}\label{thm:tmult=amult-R}
let $S$ be a parametric semi-line of $\cJ_f$, and assume that $S$ has an endpoint $q$. Then $q$ belongs to a component of the algebraic curve $\cD_{f}$ not containing $S$.
\end{theorem}

\begin{proof}
Let $\Gamma$ be an edge of $\Delta$. We keep the same notation as in this Section for $U$, $g$, and $G$ for $\KK=\RR$. We use $\CC G$ to denote the complex surface in $(\CC^*)^2\times\CC^2$ given by $\VV(g)$, and by $\overline{\CC G}$ its Zariski closure in $\CC^4$. 

Eventhough the closure $\overline{G}\subset \RR^4$ is smooth over all points in $G$, the same cannot be said about its boundary $\partial G:= \overline{G}\setminus G= \overline{G}\cap \{z_1\cdot z_2 = 0\}$. Then, using a sequence of blow-ups of $\CC^4$ at the complex algebraic curves $\overline{\CC G}\cap \{z_1\cdot z_2 = 0\}$, if necessary, we may assume (with abuse of notation) that $\overline{G}$ is smooth. The set $\overline{G}\cap\{z_2=0\}$ is now a union of real irreducible algebraic curves and of points. 

Then, we deduce from Thm.~\ref{thm:Jf=union-Tmult} that there exists an irreducible curve $\cC$ from $\overline{G}\cap \{z_2 = 0\}$, such that $S=\pi(\cC)$, where $\pi:(z,y)\mapsto y$. 

Recall that $q$ is an endpoint of $S$. Then, if $\pi_{|\cC}$ denotes the restriction of $\pi$ to the curve $\cC$, we conclude that there exists a point $p$ in its critical locus $\Crit(\pi_{|\cC})$ such that $\pi(p) = q$. In fact, since $\cC\subset \overline{G}$, the point $p$ is in the critical locus $\Crit(\pi_{|\overline{G}})$ of $\pi$ restricted to $\overline{G}$. 

Thanks to Lemma~\ref{lem:jacobian-preserved}, it is enough to show that $\pi(\Crit(\pi_{|\overline{G}}))$ contains a curve $D$ satisfying 
\[
q\in D\not\subset\pi(\overline{G}\cap\{z_2 = 0\}).
\] Since $\cC$ is an algebraic curve and $\pi(\cC)$ is a semi-algebraic curve with an endpoint $q$, it holds that $|\pi^{-1}(y)\cap \cC|\geq 2$ for finitely-many $y\in \pi(\cC)$. Let $\cU$ be a neighborhood of $p$, and let $K$ be a connected component of $G\cap \cU$, adjacent to $\cC$. Since $\pi_{|\overline{G}}$ is dominant, for all $y\in \pi(K)$, the preimage $\pi_{|K}^{-1}(y)$ also has at least two points. It follows that
\[
\emptyset \neq\pi( \Crit \pi_{|K}) \subset \pi(\Crit(\Pi_{|G})).
\] Here, the non-emptiness follows from $\pi_{|K}$ being a proper, locally algebraic morphism over a connected set $K$, whereas the inclusion is obvious.
\end{proof}

\begin{cor}\label{cor:segment-decomposition}
With the same notations as in Theorem~\ref{thm:tmult=amult-R}, let $\Sigma\subset \RR^2$ be the finite set of points formed by endpoints of $S$ and its singularities. Then $S$ is homeomorphic to a line if $\Sigma$ is empty. Otherwise, each connected component of $S\setminus\Sigma$ is a segment homeomorphic (in the induced Euclidean topology) to the open interval $]0,b[$, where $b\in\{1,+\infty\}$. 
\end{cor}

\begin{theorem}\label{thm:special_case-odd}
Let $g:=\CC f:\CC^2\to\CC^2$ be the complexification of $f$, and let $\AFCgr$ be a component of the multiplicity set $\AFCg$ of $\tilde{f}$ having odd weight (see Lemma~\ref{lem:weight}), for some $\rho\in\RR^*$ and some $\Gamma\prec\Delta$.  
%If $k=0$, or if $\AFCgr$ shares no components with $f(\{x_1\cdot x_2 = 0\})$, then 
Then, it holds that 
\[
\AFCgr\cap \RR^2\subset \cJ_f.
\]
\end{theorem}

\begin{proof}
Let $S$ denote the component $\AFCgr$. We retain the set $\Sigma\subset S$ from Lemma~\ref{lem:weight}, and define $\cV:=S\setminus \Sigma$. Then, the set $\RR\cV:=\RR^2\cap\cV $ is (homeomorphic to) a disjoint union of open intervals in $S\cap\RR^2$. Prop.~\ref{prop:tmult=amult-C} shows that for any $q\in \RR\cV$, there exists a point $y\in \RR^2\setminus S$, converging to $q$, and a set $K_y$ of points $z\in (\CC^*)^2$, converging to a subset $K_q\subset\CC^*\times\{0\}$, such that $\ous(f-y)(z) = 0$. Lemma~\ref{lem:weight} shows that, there exists $m\in\NN$ such that $\# K_q = m$ if $\Gamma$ is a semi-origin edge, and $K_q $ is a point of multiplicity $m$ otherwise. Therefore, we get $m=\# K_y$.

Recall that $\ous (f - y) =0 $ is a real polynomial system for each $y\in \RR^2$ above, and $K_q$ has at least one point in $\RR^*\times\{0\}$. Then, if $z$ is a point from $K_y$ converging to $K_q\in\RR^2$, so does its complex conjugate $\overline{z}$. Since $m=\#K_y$ is odd, at least one of the points in $K_y$ is real. We get $q\in\TFRgr$, and thus $q\in S_f$ from Thm.~\ref{thm:Jf=union-Tmult}.\end{proof}

%The following result could also be proven using topological arguments.

\begin{cor}\label{cor:birational}
Let $f:\RR^2\to\RR^2$ be a polynomial map such that its complexification $\tilde{f}:=\CC f:\CC^2\to\CC^2$ is birational. Then, the non-properness set $\cJ_f$ is a union of algebraic curves in $\RR^2$.
\end{cor}

\begin{proof}
Since the map is birational, for any edge $\Gamma$ of the Newton polytope $\Delta$ of $f$, each non-empty component $S$ of $\AFCg$ has odd weight.  Thm.~\ref{thm:special_case-odd} shows that $\cJ_f$ contains each component $S\cap\RR^2$, and Thm.~\ref{thm:Jf=union-Tmult} shows that $\cJ_f$ is formed by all such components $S\cap\RR^2$. The proof thus follows from Remark~\ref{rem:irred_real-semi}(1).
\end{proof}

}

\section{Strategy for computing the Jelonek set}\label{sec:main-alg}

In what follows $f = (f_1, f_2) :\KK^2\to\KK^2$, where
$f_1, f_2 \in \ZZ[x_1, x_2]$ are as in \eqref{eq:f1-f2}; we consider integer
coefficients to study the bit complexity of the corresponding algorithms. Also,
we consider $\Delta_1 = \NP(f_1)$, $\Delta_2 = \NP(f_2)$, and $\Delta = \Delta_1 \oplus \Delta_2$.

As a preproccessing step 
we
consider the pair $(f_1+a_1,f_2+a_2)$, where $a:=(a_1,a_2)\in\ZZ^2$ is a random
point. This is to ensure that they both have a constant term, or in other words
both their Newton polytopes contain the origin. Consequently, the output of
\sjalg will differ from the "true" Jelonek set of $f$ by
exactly this translation.

Accordingly, in what follows we assume that polynomials have a constant term.

The algorithm has two phases. During the first phase, which we detail in
Sec.~\ref{subsec:detect-faces}, we compute the Minkowski sum of the Newton
polytopes of the input polynomials and we characterize its edges, following
Def.~\ref{def:various-faces}. The partition of the edges with respect to the
various types appears in Fig.~\ref{fig:edge-partition}.
Following Lem.~\ref{lem:short} we have to consider only the
semi-origin and pertinent edges for our computations. At the second phase, which
we detail in Sec.~\ref{subsec:multipl-sets}, for each edge of interest, we
compute with a restricted polynomial system. For each, real or complex depending
on $\KK$, solution of this system, we exploit its multiplicity to compute the
corresponding \tmult set, using Prop.~\ref{prop:tmult=amult-C} or Cor.~\ref{cor:segment-decomposition}. The union of the \tmult sets is
the Jelonek set (Thm.~\ref{thm:Jf=union-Tmult}).

\begin{algorithm2e}[thp]
  \footnotesize
  %\scriptsize \dontprintsemicolon \linesnumbered
  \SetFuncSty{textsc} \SetKw{RET}{{\sc return}} \SetKw{OUT}{{\sc output \ }}
  \SetKwInOut{Input}{Input}
  \SetKwInOut{Output}{Output}
    \SetKwInOut{Require}{Require}
    % \SetVline
     \Input{$f_1, f_2 \in \ZZ[x_1, x_2]$, $\KK \in \{\CC, \RR \}$}
    \Require{Both $f_1$ and $f_2$ have a non-zero constant term. }
    \Output{The Jelonek set $\cJ$.}

    \BlankLine

    $\cJ \gets \varnothing$ \;
    $\Delta_1 \gets \NP(f_1) \, ; \, \Delta_2 \gets \NP(f_2) $ \;

    $\Delta \gets \textsc{Minkowski\_sum\_2}(\Delta_1, \Delta_2)$ \; \label{alg:sj-l-MS}
%     \tcc{Let the vertex set of $Q$ be $\{q_1,  \dots, q_{\ell}\}$.}

    \BlankLine
    \ForAll{semi-origin edges $\Gamma \prec \Delta$}{ \label{alg:sj-l-so}
      \tcc{It holds $\Gamma = \Gamma_1 \oplus \Gamma_2$}

      $f_1|_{\Gamma_1} \gets  x^{u} \cdot
      \sum_{i=0}^{|\Gamma_1\cap\NN^2|} a_i ( x_1^k x_2^l)^i $ ;
      $\quad f_2|_{\Gamma_2} \gets  x^{v} \cdot
      \sum_{j=0}^{|\Gamma_2\cap\NN^2|} b_j (x_1^kx_2^l )^j $ \;

      \lIf{$0 \not \in \Gamma_1$}{ \label{alg:sj-l-so-eq-1}
        $\cJ \gets \cJ \cup
        \left\{ (y_1, y_2) \in \KK^2 \,|\, y_2 = \sum_j b_j t^j,
          \text{ with } P_1(t) = \sum_i a_i t^i = 0 \right\}$
      }

      \lIf{$0 \not \in \Gamma_2$}{ \label{alg:sj-l-so-eq-2}
        $\cJ  \gets \cJ \cup
        \left\{ (y_1, y_2) \in \KK^2 \,|\, y_1 = \sum_i a_i t^i,
          \text{ with } P_2(t) = \sum_j b_j t^j = 0 \right\}$
      }

      \lIf{$0 \in \Gamma_1$ and $0 \in \Gamma_2$}{ \label{alg:sj-l-so-eq-3}
        $\cJ \gets \cJ \cup
        \left\{ (y_1, y_2) \in \KK^2 \,|\,
         y_1 = \sum_i a_i t^i,  y_2 = \sum_j b_j t^j,
        \text{ where } t \in \KK \right\}$
        }

    } % of FOR

    \BlankLine
    \ForAll{ pertinent edges $\Gamma \in \Delta$ such that
      $f|_{\Gamma} = 0$ has solutions in $\TT$ }{  \label{alg:sj-l-pert}

      \tcc{Let $a_{1}$ be the vertex of $\Gamma$ closer to the origin. Let $v$
        be a primitive vector on $\Gamma$ starting from $a_{1}$. We compute a
        basis for the lattice $\ZZ^{2}$, $(v, w)$ that spans positively the
        polytope $\Delta - a_1$.}

      $(v, w) \gets \textsc{Compute\_Lattice\_Basis}$ \; \label{alg:sj-l-basis}

      \tcc{It holds $T  =
        \big(\begin{smallmatrix} v_{1} & v_{2}\\ w_{1} & w_{2} \end{smallmatrix}\big)
        \in \SL(2,\ZZ)$ and $D = \det(T)$. \\
        We perform the change of variables
        $x_{1} \gets z_{1}^{w_{2}/D}z_{2}^{-v_{2}/D}, x_{2} \gets  z_{1}^{-w_{1}/D}z_{2}^{v_{1}/D}$.}

        $D \gets \det(T) = v_{1}w_{2} - v_{2}w_{1}$  \tcc{Notice that $D = \pm 1$, because $v, w$ is a unimodal basis.}

        $g_1 \gets \mathtt{numer}(f_1(z_{1}^{w_{2}/D}z_{2}^{-v_{2}/D}, z_{1}^{-w_{1}/D}z_{2}^{v_{1}/D}) - y_1)  \in \ZZ[y_1, z_1, z_2]$ \; \label{alg:sj-l-g1}
        $g_2 \gets \mathtt{numer}(f_2(z_{1}^{w_{2}/D}z_{2}^{-v_{2}/D}, z_{1}^{-w_{1}/D}z_{2}^{v_{1}/D}) - y_2)  \in \ZZ[y_2, z_1, z_2]$ \; \label{alg:sj-l-g2}

        \BlankLine
        $g \gets \gcd( g_1(y_1, z_1, 0), g_2(y_2, z_1, 0)) \in \ZZ[z_1]$ \; \label{alg:sj-l-g}

        \BlankLine
        $h \gets \det( \Jac_{z}(g_1, g_2)) \in \ZZ[y_1, y_2][ z_1, z_2]$ \; \label{alg:sj-l-Jac}

        \BlankLine

        $\cJ \gets \cJ \, \cup \, \textsc{cim\_by\_resultant}(g_{1}, g_{2}, g, \KK)$ \;\label{alg:sj-l-last}

    %     \BlankLine
    %     \ForEach{complex (or real depending on $\KK$) root $\rho$ of $g$}{  \label{alg:sj-l-rho}
          % \If{$h(\rho, 0) \not= 0$}{ \label{alg:sj-l-Jac-0}
          %    \tcc{If the Jacobian is not zero, then the multiplicity is one.}
          %    \tcc{It holds $h(\rho, 0) \in \ZZ[\rho][y_1, y_2]$.}
          %    $\cJ \gets \cJ \cup \{ h(\rho, 0) \} $ \;
          %  }
    %       \Else { \label{alg:sj-l-Jac-not-0}
    %         \tcc{The Jacobian is zero; the solution $(\rho, 0)$ has multiplicity $> 1$.}
    %         $\mathcal{J} \gets \mathcal{J} \cup
    %         \textsc{cim\_by\_resultant}(g_{1}, g_{2}, (\rho, 0))$ \;
    %         \tcc{Alternatively, we could use
    %           $\textsc{cim}(g_{1}, g_{2}, (\rho, 0))$.}
    %       }
    %     }
      }
      \RET $\cJ$\;

    \caption{\sjalg$(f = (f_1, f_2), \KK)$}
  \label{alg:sparse-Jelonek-2}
\end{algorithm2e}

\subsection{Characterizing the edges}\label{subsec:detect-faces}

The first phase of \sjalg characterizes the edges (and the corresponding
summands), see Sec.~\ref{subs:long-short}, of the Minkowski sum
$ \Delta = \Delta_1 \oplus \Delta_2$.
In particular, it identifies the infinity edges
$\Gamma \prec \Delta$ which are either semi-origin or
pertinent. The characterization relies on the summands of $\Gamma$, say
$(\Gamma_1, \Gamma_2)$. 
The algorithm for computing the Minkowski sum should check whether $\dim(\Gamma_i)$
is zero or not and if $\Gamma_i$ contains $\{0\}$ or not, for $i \in \{1,2\}$,
see Def.~\ref{def:various-faces}.

\subsection{Computing the multiplicity sets}\label{subsec:multipl-sets}

We consider only the semi-origin and the
pertinent edges (Lem.~\ref{lem:short}). For each one them, say
$\Gamma$ with summands $(\Gamma_1, \Gamma_2)$, the algorithm computes the \amult
set, \AFKg, which is a superset of the \tmult set \TFKg. The \amult set consists
of the set of points $y\in\KK^2$ for which there exists a particular
transformation $U\in \SL(2,\ZZ)$ (see Section~\ref{subsec:base-change}) such
that $\ous (f -y) = 0$ has more isolated solutions (counted with multiplicities)
in $\KK^*\times\{0\}$ than $\ous f = 0$ has in $\KK^*\times\{0\}$ (see
Defs.~\ref{def:t-face-mult} and~\ref{def:a-face-mult}).

First, we consider the  semi-origin edges,
Line~\ref{alg:sj-l-so} in \sjalg (Alg.~\ref{alg:sparse-Jelonek-2}). For each such $\Gamma$,
we deduce from Remark~\ref{rem:one-var} that after we restrict $f -y$ to $\Gamma$,
we obtain univariate polynomials in  a monomial  $x_1^k x_2^l$ for some $k, l \in \NN$;
which in turn we consider it as new variable $t$. 

Thus,  $(f_1 - y_1)_{\Gamma} \in \ZZ[y_1, t]$ and
$(f_2 - y_2)_{\Gamma} \in \ZZ[y_2, t]$,
which is a parametric representation of $\TFKg$;
see Lines~\ref{alg:sj-l-so-eq-1},
\ref{alg:sj-l-so-eq-2}, and \ref{alg:sj-l-so-eq-3} of
\sjalg.
We can obtain the implicit equation $\TFKg$ by eliminating the common variable
$t$ from the system. The multiplicity set is a point, a parametrized curve, or
a union of lines.

Next, we consider the pertinent edges, Line~\ref{alg:sj-l-pert} of
\sjalg. Let $\Gamma\prec \Delta$
be pertinent.
First, we transform the polynomials $f - y$, using a toric change of variables,
to $g = (g_1, g_2) \in \QQ[y_1, y_2][z_1, z_2]$, Lines~\ref{alg:sj-l-g1} and
\ref{alg:sj-l-g2}, i.e., $g := \ous(f-y)$.
 The change of variables corresponds to a 
 $\Gamma$-basis for $A$ (see Sec.~\ref{subsec:base-change}.)
  of the lattice
$\ZZ^2$, Line~\ref{alg:sj-l-basis}.

Following Remarks~\ref{rem:Ustar-0dim},~\ref{rem:one-var}, and Lem.~\ref{lem:short},
the change of basis preserves the number of the 0-dimensional solutions
and results in a simpler system to solve; actually a univariate one, Line~\ref{alg:sj-l-g}.

We consider $\Gamma$ only if $\ous f_\Gamma =0$ has solutions in
$\KK^*\times\{0\}$; 
(see Remark~\ref{rem:one-var} and
Lem.~\ref{lem:short}). This excludes edges $\Gamma$ that are not
semi-origin or not pertinent; see Lem.~\ref{lem:short}. Therefore, $g=0$ 
 has solutions $(\rho, 0) \in \KK^*\times\{0\}$ whose multiplicity
changes depending on whether $y$ belongs to the \amult set or not
(Def.~\ref{def:t-face-mult}). For each above solution $\rho$ we compute a bivariate
polynomial $J_{\rho} \in \ZZ[\rho][y_1, y_2]$ the zero set of which is a curve
$\cK_{\rho} := \VV_{\KK}(J_{\rho}) = \AFKgr \subset \KK^2$, which the \amult
set.

If $\KK = \CC$, then $\TFCg = \AFCg$ and the \tmult set  is the union of all
$\AFCgr$; Prop.~\ref{prop:tmult=amult-C}.
If $\KK = \RR$, then we have to test whether certain distinguished

points, say $\{p_1,\ldots,p_r\}$,
in $\AFRgr$ (if any) are in the Jelonek set or not. That is, according to Cor.~\ref{cor:segment-decomposition}, we subdivide $\AFCgr\cap\RR^2$ into $r+1$ open intervals separated by points in $\cD_f\cap\AFCgr$, $\Sin(\AFCgr)$, and so to avoid the discriminant and all other multiplicity sets.

If $p_i$ belongs to the Jelonek
set, then $\TFRg$ contains the interval $\ell_i\subset \AFCgr\cap\RR^2$. The union of all such intervals forms $\TFRg$. If $p_i$ is not in the Jelonek set, then we ignore this line segment.

We present Alg. \textsc{ms\_resultant} in Sec.~\ref{sec:G-mult-resultant} for computing the multiplicity of a solution of a
bivariate system and the corresponding \tmult sets.

\subsection{Complexity and representation of the output}
\label{subsec:complexity}

Let $f_1, f_2 \in \ZZ[x_1, x_2]$ be polynomials of degree $d$ and bitsize
$\tau$. Also, assume that their Newton polytopes have at most $n$ edges. We go
over the various steps of \sjalg to estimate their complexity.

\medskip
\noindent
\emph{Initial computations.}
First, we compute the Minkowski sum, Line~\ref{alg:sj-l-MS}. This costs
$\sOO(n)$ and results a polygon with at most $\OO(n)$ edges at its convex hull.
We do this by slighly modifying the  well known
optimal algorithm \cite[Sec.~13.3]{BCKO-CG-08} for computing the Minkowski sum
of two polygons to remember the summands of its edge of the sum.

\medskip
\noindent
\emph{Computations with semi-origin edges.}
Next, we consider the semi-origin edges of the Minkowski
sum, Line~\ref{alg:sj-l-so} to \ref{alg:sj-l-so-eq-3}; that is edges having at
least one summand containing $0$.
The most computational expensive operation is the computation of the roots (real
or complex) of  univariate polynomials; these are the polynomials  $P_1(t)$ and
$P_2(t)$ appearing at Lines~\ref{alg:sj-l-so-eq-1} and \ref{alg:sj-l-so-eq-2}.
They come from the restriction of $f_1$ or $f_2$ on $\Gamma$ and so
their degree is at most $d$ and bitsize at most $\tau$. The computation of
their roots (real or complex) costs $\sOB(d^3 + d^2\tau)$ \cite{Pan-usolve-02}.
As there are at most $\OO(n)$ semi-origin edges, the total cost of the first
phase is $\sOB(n(d^3 + d^2\tau))$.

Regarding the output of this phase, first we assume  that  $\KK = \RR$.
If the real roots of $P_1 \in \ZZ[t]$ are
$\gamma_1, \dots, \gamma_r$, then whenever $0 \not\in \Gamma_2$, the output is a union of horizontal lines
defined by numbers in $\ZZ[\gamma_i]$, for $i \in [r]$.
Similarly for the case $0 \not\in \Gamma_1$.
When both $\Gamma_1$ and $\Gamma_2$ contain 0 (Line~\ref{alg:sj-l-so-eq-3}),
then $\cJ_f$ is a parametrized polynomial curve, defined by polynomials
in $\ZZ[t]$ of degree at most $d$ and bitsize at most $\tau$. For $\KK=\CC$, its implicit
representation, consists of a polynomial in $\ZZ[y_1, y_2]$ of
degree at most $d$ and maximum coefficient bitsize $\sOO(d \tau)$. We compute it
in $\sOB(d^3\tau)$, e.g., \cite{LicRoy-sub-res-01}.
If $\KK = \CC$
and $0 \in \Gamma_1$ (or $0 \in \Gamma_2$), then
we represent the union of lines in a more unified way
by considering the resultant
$R_2(y) =  \res( y_2 - \sum_j b_j t^j, P_1(t), t) \in \ZZ[y_2]$.
In this way we have the implicit representation  $(y_1, R_2(y))$ for $\cJ_f$ whenever $\KK=\CC$ and a parametrized one when $\KK = \RR$.

\medskip
\noindent
\emph{Computations with pertinent edges.}
The last phase of the algorithm deals with pertinent edges.
The first
task consists in computing a unimodular basis that fits our needs
Line~\ref{alg:sj-l-basis}
and Section~\ref{subsec:base-change}.
The most costly part of this procedure is the computation of the Smith Normal
Form (SNF) of matrix. As the degree of the input polynomials is $\OO(d)$, this
also bounds the dimension of the matrices. The cost of SNF is
$\sOB(d^{\omega +1})$ \cite{Stor-SNF-96}, where $\omega$ is the exponent of
matrix multiplication.

After the toric change of variables
we obtain a univariate polynomials and we compute its roots.
Then, we compute the  multiplicity sets
by exploiting the multiplicities of the roots of the corresponding bivariate polynomial system.
\msresultant (Alg.~\ref{alg:cim-resultant}) computes the corresponding
multiplicity sets in $\sOB(d^7 + d^6\tau)$ for the complex case
and  $\sOB(d^{19.11} + d^{18.11}\tau)$ in the real case.
As there at most $n$ pertinent edges, we should multiply the previous bounds by $n$
to get the overall cost.
The latter bounds dominate the complexity of the algorithm.

\begin{theorem}\label{th:sparse-complexity}
  Let $f = (f_1, f_2) : \KK^2 \to \KK^2$ be a dominant polynomial,
  where $f_1, f_2 \in \ZZ[x_1, x_2]$ are polynomials of size  $(d, \tau)$
  and their Newton polytopes have at most $n$ edges.
  \sjalg computes the Jelonek set
  of $f$ in
  $\sOB(n(d^7 + d^6\tau))$ over the complex
  and  in $\sOB(n(d^{19.11} + d^{18.11}\tau))$ over the real numbers.
\end{theorem}

If the polynomials are generic, then the complexity bound becomes significantly
better.

\begin{cor}\label{cor:sparse-complexity}
  Let $f = (f_1, f_2) : \KK^2 \to \KK^2$ be a dominant polynomial,
  where $f_1, f_2 \in \ZZ[x_1, x_2]$ are polynomials of size  $(d, \tau)$
  and their Newton polytopes have at most $n$ edges.
  If $f_1$ and $f_2$ are generic, then
  \sjalg computes the Jelonek set of $f$ in
  $\sOB(n(d^3 + d^2\tau))$.
\end{cor}
\begin{proof}
  When the input polynomials are generic, then
  the Minkowski sum of their Newton polytopes does  contain a pertinent edge
  with probability 1.
  Therefore, the complexity of the algorithm depends on the computation of the
  Minkowski sum and the manipulation of the semi-origin edges. The latter
  dominates the complexity bound. Its complexity is $\sOB(n(d^3 + d^2\tau))$, as
  we have to solve $n$ times a univariate polynomial.
\end{proof}

\begin{remark}
  The genericity property of the polynomials in Cor.~\ref{cor:sparse-complexity} is
  with respect to fixed Newton polytopes and with respect to non-fixed ones.
\end{remark}

\subsection{An example}\label{subsec:example}

Let $\Delta:=\Delta_1\oplus \Delta_2$, where $\Delta_1$, and $\Delta_2$ are integer polytopes in
$(\RR_{\geq 0})^2$; they appear at the left hand side of Figure~\ref{fig:pol}. The right hand side 
illustrates $\Delta$.

We want to compute the Jelonek set $\mathcal{J}_f$ of the map $f=(f_1,f_2):\KK^2\to\KK^2$,
where
\begin{align*}
f_1= & 1+x_1x_2 +2x_1^2x_2^2 - \tfrac{7}{10}x_1^2x_2 - 3x_1^3x_2^2,\\
  f_2= & 1+3x_1x_2 -4x_1^2x_2^2 +5x_1^3x_2^3  -6x_1^4x_2^4 + \tfrac{3^7}{2^5}x_1^{10}x_2^4  -54x_1^{6}x_2^3
         + \tfrac{5103}{320}x_1^{9}x_2^3 - x_1^7x_2^2 + x_1^4x_2.
\end{align*}
Figure~\ref{fig:pol} shows that $\NP(f) = \Delta$ has exactly six infinity edges;
they are either pertinent or semi-origin (see
Def.~\ref{def:various-faces}). Let
$S_{01},S_{02}, S_{13}, S_{24}, S_{35}, S_{46}\subset\KK^2$, denote the
corresponding multiplicity sets in $\KK^2$ (see
Def.~\ref{def:t-face-mult}) where
\[
  S_{ij}:=\TFK(a_i\oplus b_j).
\]
By Thm.~\ref{thm:Jf=union-Tmult}, the union of the latter is $\cJ_f$.
We apply \sjalg to compute $\cJ_f$.
\begin{itemize}
  \item Edges $a_0\oplus b_1$, $a_0\oplus b_2$, and $a_1\oplus b_3$ are infinity
    semi-origin an only $a_0$ contains $(0,0)$. Therefore,
    Lines~\ref{alg:sj-l-so-eq-2} and~\ref{alg:sj-l-so-eq-3} of
    \sjalg results
	\begin{align*}
		S_{01}  = & \{(y_1, y_2)\in\KK^2~|~y_1=1,~1 +t =0\}  =  \{y_1 = 1\},  \\
		S_{02}  = & \{(y_1, y_2)\in\KK^2~|~y_1=1,~1 -t =0\}  =  \{y_1 = 1\},  \\
		S_{13}  =  & \{(y_1, y_2)\in\KK^2~|~y_1=1-7t/10,~320 + 5163t =0\}  = \{y_1 - 761/729 = 0\}.
	\end{align*}

	\item The edge $a_4\oplus b_6$ is origin, where $(0,0)\in a_4$ and $(0,0)\in b_6$. Then, Line~\ref{alg:sj-l-so-eq-3} results
	\begin{align*}
		S_{46} = & \{(y_1, y_2)\in\KK^2~|~y_1=1+t+2t^2,~y_2=1 +3t-4t^2+5t^3-6t^4\}.
	\end{align*}	%10 - alg:sj-l-basis, 15 - alg:sj-l-Jac, 17 - alg:sj-l-Jac-0, 19 - alg:sj-l-Jac-not-0
  \item Edges $a_2\oplus b_4$, and $a_3\oplus b_5$ are pertinent. Moreover, the
    systems $f|_{a_2\oplus b_4 }=0$ and $f|_{a_3\oplus b_5 }=0$ have solutions in
    $\TT$. Then, we apply  Lines~\ref{alg:sj-l-basis} --~\ref{alg:sj-l-last}
    to $f$ for both edges. The algorithm outputs
	\begin{align*}
      \AFK(a_2\oplus b_4) =  \{y_1+1468/18225 =0\}
      \text{ and }
      \AFK(a_3\oplus b_5)  =  \{y_1+6238/10935 =0\},
	\end{align*}
    and the reader can check that $\TFK(a_2\oplus b_4)$
    and $\TFK(a_3\oplus b_5)$ are non-empty when $\KK=\RR$.
  \end{itemize}

\section{Computing the multiplicity sets}
\label{sec:mult-set-computation}

Let $f:\KK^2\to\KK^2$ be a dominant polynomial map sending $(0,0)$ to $\TT$
and
$\Delta:=\NP(f)$. In this
section we present how to compute the \tmult set,  \TFKg,
corresponding
to a pertinent edge $\Gamma\prec \Delta$.

Following the process of  \sjalg,
for a pertinent edge $\Gamma$, after we perform a toric change of variables,
we obtain a bivariate polynomial
system that has solutions of the form $(\rho, 0)$, for some complex (or real)
$\rho$, Line~\ref{alg:sj-l-g}. Each $\rho$ gives rise to the set
\AFKgr, where 
\[
  \TFKg \subseteq \AFKg = \bigcup_{\rho\in\SolKg}\AFKgr.
\]

The computation of $\AFKgr$ depends on the multiplicity of $(\rho, 0)$. We
present an algorithm \msresultant, for computing the multiplicity of $(\rho, 0)$ as a root
of a bivariate polynomial system and then use it to compute the multiplicity set
$\AFKgr$. The algorithm exploits resultant computations and its
worst case complexity is polynomial.

\subsection{Resultant computations for the multiplicity set}
\label{sec:G-mult-resultant}

The pseudo-code of the algorithm appears in Alg.~\ref{alg:cim-resultant} and its
proof of correctness presents the details of the various steps.

\begin{prop}\label{prop:correct-resultant}
  \msresultant, Alg.~\ref{alg:cim-resultant}, correctly computes the \amult set
  \AFKg that corresponds to a \emph{pertinent} edge $\Gamma\prec \Delta$. The result
  is a (possibly empty) union $ \bigcup_{\rho} \cK_{\rho}$ of finitely-many
  curves $\cK_{\rho}$ that are the zero loci of polynomials in $\KK[y_1,y_2]$,
  where $\KK \in \{\RR, \CC\}$ and $\rho$ runs over all
  points in $\KK^*$.
\end{prop}
\begin{proof}
  Let $(\rho, 0)$ be a solution of $g_1  = g_2 = 0$,
  say of multiplicity $\mu$, where $g_i \in \ZZ[y_1, y_2][z_1, z_2]$.
  We will exploit the fact that if we project on $z_1$, then $\rho$
  is a root of the resultant of multiplicity at least $\mu$.

  We consider the resultant of $g_1$ and $g_2$ with respect to $z_2$. This
  results a polynomial $R_1 \in \ZZ[y_1, y_2][z_1]$. Obviously, $\rho$ is a root
  of $g$ of multiplicity at least $\mu$, say $\mu_1 \geq \mu$.
  We also know that $R_1$  factors as
  $R_1 = R_{11} R_{12}$, where $R_{11} \in \ZZ[z_1]$ and
  $R_{12} \in \ZZ[y_1, y_2][z_1]$. The presence of the factor $R_{11}$ is guaranteed because the
  system has solutions of the form $(\rho, 0)$.
  We divide out the factors of $R_1$ that depend only on $z_1$, that is $R_{11}$,
  and we end up with the polynomial $R_{12}$.
    If we do the substitution $z_1 = \rho$ to $R_{12}$, then the resulting
  polynomial $J_1$  is in $\ZZ[\rho][y_1, y_2]$ and it is non-zero.
  If we want $\rho$ to be of higher multiplicity as a root of $R_1$,
  then $J_1$ should be zero.
  Thus, $J_1$ is a superset of the part of the multiplicity set $\AFKgr$ 
  that emanates from $(\rho, 0)$.
  Then, we consider the resultant of $g_1$ and $g_2$
  with respect to $z_1$.
  We proceed as before, mutatis mutandis,
  where now the $R_{21}$ is a power of $z_2$.
  At the end we compute
  a superset of the part of the multiplicity set $\AFKg$
  emanating from $(\rho, 0)$.

  If we want the solution $(\rho, 0)$ to have multiplicity higher than $\mu$ as
  a solution to the system $g_1 = g_2 = 0$, then both $\rho$ and $0$ should be
  roots of higher multiplicity of the corresponding resultants, that is $R_1$
  and $R_2$, respectively. Thus, the gcd of $J_1$ and $J_2$ should vanish. The
  latter is $J_{\rho} \in \ZZ[\rho][y_1, y_2]$
  which is the implicit equation of a  curve $\cK_{\rho} = \AFKgr$
  that is the \amult set.
  For the complex case, the \amult set $\AFCgr$ is the \tmult set
  $\TFCgr$. For the real case, following Cor.~\ref{cor:segment-decomposition}, we have to
  check if a line segment of $\AFCgr\cap \RR^2$ belongs to $\TFRgr$ or not. 
%  If it is not empty,
%  then it equals $\AFRgr$.

  By repeating the same  procedure over all solutions $(\rho, 0)$ of  the system
  $g_1 = g_2 = 0$,  we obtain the \tmult set $\TFRg$.
\end{proof}

\subsubsection{The complexity of the complex \tmult set}

\begin{theorem}[Complexity of \msresultant over $\CC$]
  \label{thm:cim-res-complexity-C}
  Consider the polynomials $g_1, g_2 \in \ZZ[y_1, y_2][z_1, z_2]$ which are of
  degree $d$ with respect to $z_1$ and $z_2$ and of degree $1$ with respect to $y_1$ and $y_2$
  and bitsize $\tau$.
  Let $g \in \ZZ[z_1]$ be of size  $(d, \tau)$.
  The bit complexity of $\msresultant(g_1,g_2,g,\CC)$,
  Alg.~\ref{alg:cim-resultant}, is $\sOB(d^7 + d^6\tau)$.
\end{theorem}
\begin{proof}
  The resultant $R_1$ is a polynomial in $(z_1, y_1, y_2)$ of degree
  $(\OO(d^2), \OO(d), \OO(d))$ respectively \cite[Prop.~8.49]{BPR03}, and bitsize $\sOO(d \tau)$
  \cite[Prop.~8.50]{BPR03}.
  To compute $R_1$ we employ fast subresultant algorithms, e.g., \cite{LicRoy-sub-res-01}.
  We perform $\sOO(d)$ operations.
  Each operation consists of multiplying two trivariate polynomials
  in $z_1, y_1, y_2$; with degrees and bitsize as mentioned before.
  Each multiplications costs $\sOB(d^5\tau)$
  and so the overall costs for computing $R_1$ is $\sOB(d^6\tau)$.

  To compute $R_{12}$ we consider $R_1$ as a bivariate polynomial in $y_1$ and $y_2$
  with coefficients in $\ZZ[z_1]$ and we compute its primitive part;
  that is to compute the gcd of all the coefficients and then divide all of them with it.
  The coefficients are polynomials in $z_1$ of degree $\sOO(d^2)$ and bitsize $\sOB(d \tau)$.
  We can compute their common gcd in
  $\sOB(d^6 + d^5\tau)$ with Las Vegas algorithm \cite[Lem.~2.2]{KRTZ-ptopoo-20}.
  The cost of the exact division is dominated by the cost of this operation.
  The polynomial $R_{12} \in \ZZ[y_1, y_2][x_1]$
  has degree(s) $(\OO(d^2), \OO(d), \OO(d))$ and bitsize $\sOO(d^2 + d \tau)$.

  The same complexity and bitsize  bounds hold for the computation of $R_2$ and $R_{22}$.

  The polynomial $g$ is of size $(d, \tau) $ and so the cost for solving is
  $\sOB(d^3 + d^2\tau)$ \cite{Pan-usolve-02}. Then, for each root of $g$, say
  $\rho$, we make the substituion $(z_1, z_2) \gets (\rho, 0)$ to obtain the
  polynomials $J_1$ and $J_2$ and we compute their $\gcd$, that is
  $J = \gcd(J_1, J_2) \in \ZZ[\rho][y_1, y_2]$. The latter is of degree $d$ with respect 
  to both $y_1$ anf $y_2$. Its coefficients are polynomials in $\rho$ of bitsize
  $\sOO(d^3 + d^2\tau)$. The expected cost of the $\gcd$ is
  $\sOB(d^6 + d^5 \tau)$ \cite{Langemyr-single-90}, that is almost linear in the
  size of the output. As we have to do this at most $d$ times in the worst case,
  which corresponds to the different roots of $g$, the bounds follows for the
  complex case.
\end{proof}

\begin{remark}  
  For the complex case, we can avoid working with algebraic numbers. We  notice
  that we perform computations with all the roots of $g$
  and so
  we exploit the Poisson formula of the resultant. 
  Thus, we first consider the resultant of $R_{12}$ and $g$,
  that is $H \gets \res(R_{12}, g, z_1) \in \ZZ[y_1, y_2]$;
  this is the evaluation of $R_{12}$ over all the roots of $g$.
  Then, we consider the $\gcd$ with $R_{22}$, that is
  $\gcd(R_{22}, H) \in \ZZ[y_1, y_2]$.

  The theoretical complexity is same as considering each root independently
  and perform $\gcd$ computations in an extension field.
  The gain is that all the output polynomials have integer coefficients.
\end{remark}

\begin{algorithm2e}[ht]
  %\scriptsize \dontprintsemicolon \linesnumbered
  \SetFuncSty{textsc} \SetKw{RET}{{\sc return}} \SetKw{OUT}{{\sc output \ }}
	\SetKwInOut{Input}{Input}
	\SetKwInOut{Output}{Output}
    \SetKwInOut{Require}{Require}
    \Input{$(g_1, g_2) \in (\ZZ[y_{1} y_{2}])[z_{1}, z_{2}]$,
      $h \in \ZZ[z_1]$, $\KK \in \{\RR, \CC\}$ }
    \Require{ The  $g_1, g_2$ are the polynomials  we obtain
      after we apply a toric change of variables to $f_1, f_2$, that
      corresponds to a pertinent edge $\Gamma$ of the Minkowski sum
      $\NP(f_1) \oplus \NP(f_2)$.
      \newline
      The distinct roots of $h$ are $\rho_1, \dots, \rho_r$, and
      $(\rho_i, 0)$ is a solution to the system $g_1 = g_2 = 0$, for all
      $i \in [r]$.
      Notice that
      $h(z_1) = c \prod_{i=1}^{r}(z_1 - \rho_i)^{\mu_i}$, where $c \in \ZZ$. }
    \Output{The multiplicity set $\Mfg$ corresponding to a pertinent edge $\Gamma$.}

    \BlankLine

    $R_1 \gets \res(g_1, g_2, z_2)$ \;
    \tcc{It holds $R_1 = R_{11}R_{12}$,  $R_{11} \in \ZZ[z_1]$, $R_{12} \in \ZZ[y_1, y_2][z_1]$}
    \tcc{In particular $R_{11}(z_1) = c_1 \prod_{i=1}^{m}(z_1 - \rho_i)^{\mu_{1,i}}$
      and $c_1 \in \ZZ$.}

    \BlankLine

    $R_2 \gets \res(g_1, g_2, z_1)$ \;
    \tcc{It holds $R_2 = R_{21}R_{22}$, $R_{21} \in \ZZ[z_2]$, $R_{22} \in \ZZ[y_1, y_2][z_2]$}
    \tcc{In particular, $R_{21} = z_2^{\mu_2}$, for some $\delta_2 \in \NN$.}

%    $\tilde R_1 \gets R_1 / R_{11} \in \ZZ[y_1, y_2][z_1]$ \;     $\tilde R_2 \gets R_2 / R_{21} \in \ZZ[y_1, y_2][z_2]$ \;

    \BlankLine
    $\TFKg \gets \varnothing$ \;
    
    \If{$\KK = \RR$}{
      \For{ every real root $\varrho$ of $g$ }{
        $J_1 \gets subs(z_1 = \varrho, R_{12})$ \;
        $J_2 \gets subs(z_2 = 0, R_{22})$ \;
        $J_{\rho} \gets \gcd(J_1, J_2) \in \ZZ[\rho][y_1, y_2]$ \;
        Partition $\VV(J_{\rho})$ to curve segments based on its topology \; 
        Check which segments contributes to the Jelonek set \;
       }
    }
    \BlankLine
    \If{$\KK = \CC$}{
      $H \gets \res(R_{12}, g, z_1) =
      \prod_{g(\rho) = 0} R_{12}(\varrho, y_1, y_2) \in \ZZ[y_1, y_2]$ \;

      $J \gets \gcd(R_{22}, H) \in \ZZ[y_1, y_2]$ \;

      $\TFKg \gets \VV(J)$ \;
    }

    \RET $\TFKg$ \;

      \caption{\msresultant$(g_1, g_2, g, \KK)$}
  \label{alg:cim-resultant}
\end{algorithm2e}

\subsubsection{Testing the emptiness of $\TFRgr$ and its complexity}
\label{sec:test-empty-mset}

% The multiplicity set for an edge $\Gamma$, $\Mfg$, is the union of
% semi-algebraic curves (Proposition~\ref{prop:char-mult}), each of which depends on a (real) root of a univariate
% polynomial; this is the polynomial $g$ (Line~\label{alg:sj-l-g}) of the algorithm.
% We consider only the real roots of $g$, say $\rho$, to compute the real Jelonek set.
% However, in this case not all the corresponding
% curves $\cC_{\rho}$ contribute to the (real) Jelonek set. This is so
% because there might be two sequences of complex (conjugate) numbers, say
% $\{\varrho_k\}_{k \in \NN}$ and $\{ \overline{\varrho}_k\}_{k \in \NN}$, that
% converge to a real point $(\rho, 0)$. Hence, we need  a procedure to test whether
% $\cC_{\rho}$ contributes to the Jelonek set $\cJ_f$,
% or in other words to test whether $\Mfgr$ is empty or not.

Our setting is as follows. Assume that for a pertinent edge $\Gamma \in \Delta$ we
have performed the toric transformation to the system and have computed a real value
$\rho$ resulting in a polynomial in $J_{\rho} \in \ZZ[\rho][y_1, y_2]$
the zero set of which is a curve $\VV_{\RR}(J_{\rho}) = \AFRgr \subset \RR^2$. Let us call this curve $\cCr$. 
%$\VV_{\RR}(J_{\rho}) = \AFRgr =: \cK_{\rho}\subset \RR^2$.
The degree of $J_{\rho}$, with respect to $y_1$ and $y_2$, is $\OO(d)$ 
and its coefficients are polynomials in $\rho$ of bitsize $\sOO(d^3 + d^2\tau)$; also the $\rho$ is the real root of a polynomial of size $(\OO(d), \sOO(d + \tau))$.

First, we check the  condition of Thm.~\ref{thm:special_case-odd}:
We pick a point $p\in\cC_\rho$ and we  compute the value $\delta:=\mu_\rho(p)-\mu_\rho(0)$.
Due to the preprocessing procedure, recall that $0\not\in\cJ_f$.
If $\delta$ is odd, Thm.~\ref{thm:special_case-odd} shows that $\cC_\rho\subset\cJ_f$,
and we are done. The complexity of this step is dominated by the complexity of the subsequent steps, so we do not elaborate further. 

If  Thm.~\ref{thm:special_case-odd} does not apply, then
we need to consider the intersection points of $\cCr$ with the discriminant of $f$, say $\cD$.  
Recall that $\cD$ is defined by a polynomial in $\ZZ[y_1, y_2]$ and has size $(\sOO(d^2), \sOO(d^2 + d\tau))$. 
Consequently, we partition $\cC$ using these intersection points
and we test which curve segments contribute to the Jelonek set of $f$. 

To do so, we choose a generic point, say $p$, at each curve segment
and we rely on the fact that the following two conditions are equivalent:
\begin{enumerate}[(C1)]
  \item\label{it:p_generic} The point $p$ is a generic point in $\cJ_f$.
  \item\label{it:p_viscinity}
   There exists a point $q\not\in\cJ_f$ 
  outside a small-enough ball in $\RR^2$ around $p$ such that the number of real
  solutions (counted with multiplicities) of $f-p=0$ is smaller than the
  number of real solutions to $f-q=0$.
\end{enumerate}

\paragraph{Partition $\cCr$ to curve segments}

Our strategy to partition $\cCr$ is as follows. We compute the topology of $\cCr$, that is an abstract graph that is isotopic to the curve in $\RR^2$ \cite[p.~184]{BoiTei-book-2007}. Then, we further refine the graph by considering the intersection points with $\cD_f$.
This algorithm also provides points on its of the corresponding curve segments.  

To compute the topology of $\cCr$ we employ the state-of-the-art algorithm in \cite{DDRRS-curve-18}. Unfortunately, this algorithm considers polynomials with integer coefficients, which is not our case. We notice that the bottleneck in the algorithm for the computing the topology is the projection on the coordinate axis and the solutions of univariate polynomials; the solution of which correspond to the coordinates of the critical points of the curve. 
	The project $\cCr$ on the $y_1$ (or $y_2$ axis) we consider the resultant of $\cJ_\rho$, and its derivative with respect to $y_1$, we eliminate $y_2$ and we obtain a polynomial in $y_1$. 
	The projection costs $\sOB(d^6 \tau)$ \cite[Prop.~8]{det-jsc-2009}.
	This results a univariate polynomial in $\ZZ[\rho][y_1]$
	of size $(\sOO(d^2), \sOO(d \tau))$.
	We can solve this polynomial, that have coefficients in an extension field, in $\sOB(d^{11} + d^{10}\tau)$ \cite{st-sef-jsc-2019}. The latter bound dominates the complexity of computing the topology of $\cCr$. 
	
Then we need to consider the intersection points of $\cCr$ and $\cD_f$.
This costs less that computing the topology of $\cCr$ as the latter
(implicitly) requires to solve the system of $\cJ_{\rho}$ and its derivatives and it involves polynomials of higher degree and bitsize than the polynomial defining $\cD_f$.  

At this point we can assume that we have partitioned $\cCr$ to curve segments and we have computed a generic point at each, say $p$.
One the coordinates is a rational number of bitsize at most 
$\sOO(d^7 + d^6 \tau)$.

Therefore, it suffices to test condition~\ref{it:p_viscinity} for a
point $p$ in $\cC_\rho$. To do this, we perform the following steps.

\begin{enumerate}[(S1)]
   \item We compute a point $q = (q_1, q_2) \in \RR^2$ such that the segment from $p$
    to $q$ does not intersect any other curve $\cC_{\rho'}$ and the discriminant
    curve of $f$. The curves $\cC_{\rho'}$ correspond to \amult sets
    emanating from other roots $\rho'$ and/or edges of $\NP(f)$. We obtain the
    discriminant curve by eliminating $x_1$ and $x_2$ from the equations
    $\{ f_1(x_1, x_2) - y_1 , f_2(x_1, x_2) - y_2 ,  \det( \Jac(f_1, f_2)) \}$.

   \item We check whether the systems $f - p = 0$ and $f - q = 0$ have the
    same number of real solutions. If this is the case, then 
    	the curve segment of $\cCr$ that $p$ belongs to does
    	is not part of the multiplicity set $\TFRg$, following Def.~\ref{def:t-face-mult}.
\end{enumerate}
The above three steps are enough for our purposes: The set
$\RR^2\setminus\mathcal{D}_f\cup\cJ_f$ is a union of connected components, each
of which has a constant number of real preimages under $f$. A ball $B$ small
enough around a generic point $p\in\cC_\rho$ is either included in one of the
components of $\RR^2\setminus\mathcal{D}_f\cup\cJ_f$ or intersects exactly two
of them. Assume that $p\in\cD_f$.
%, we deduce  the above discussion, together with the definition of $\cD_f$ yields the following. 

%\begin{itemize}
%
%	\item[-] If $p\not\in\cD_f\cup\cJ_f$, then $N_p = N_q = N_{q'}$,
%	
%	\item[-] If $p\in\cD_f\setminus\cJ_f$, then $N_p = N_q > N_{q'}$, or $N_p = N_{q'} > N_{q}$, and
%
%\end{itemize} 
We conclude from condition \ref{it:p_viscinity} that if we sample two points $q$
and $q'$, following (S1), then they belong to two different connected components
of $\RR^2\setminus\mathcal{D}_f\cup\cJ_f$ adjacent to $\cJ_f$ if and only if the
number of real solutions (counted with multiplicities) of the system $f-p=0$ is
smaller than that of either $f-q=0$, or $f-q'=0$.

There is a straightforward way of realizing the steps to test the where the corresponding curve segment contributes to the Jelonek set by solving various polynomial systems. Instead, we will avoid solving
systems (as much as possible) and we will rely on (separation) bounds of their roots
\cite{emt-dmm-j} to compute the points of interest.

The polynomial $J_{\rho}$ has degree $\OO(d)$ with respect to $\rho$, $y_1$, and
$y_2$ and its bitsize is $\OO(d^3 + d^2\tau)$
\cite{Langemyr-single-90,Langemyr-single-90}; the latter follows from the gcd
computations in an extension field.

\paragraph{(S1) Compute the intermediate point $q$}

\begin{figure}[h]
  \centering
  \includegraphics[scale=0.9]{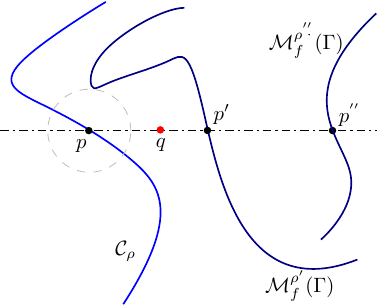}
  \caption{How to test  a curve segment of $\cC_{\rho} $ contributes to the Jelonek set.
    The horizontal line $y_1 = p_1$ intersects the other curves of other multiplicity sets. If $p'$ is the closest intersection to $p$, then $q$ is a point in the segment $p\, p'$.
  }
  
  \label{fig:compute-disc}
\end{figure}

One of the coordinates of $p$ is a rational number. Assume, without loss of generality that this is the second one and its value is $c$.
That is $p = (p_1, c)$, where $p_1$ is a root of $J_{\rho, c} \in \ZZ[\rho][y_1]$
of degree $d$ and bitsize $\sOO(d^8 + d^7\tau)$.
Consider the line $y_2 = c$ and the intersection of this line with all the other curves, say $\cC_{\rho'}$ with equation $J_{\rho'}(y_1, y_2)$ that
correspond to other multiplicity sets, and the discriminant curve of $f$.
Let the closest intersection to $p$ be a point $p'$. Then it suffices to choose
the middle point of the segment $p\, p'$.
Even more, we can estimate, using root bounds how close to $p$ the point $p'$ could be in the worst case. Then choose a point $q$ that it closest to $p$ with respect to this worst case estimation.
Using root bounds, there is no need to solve the corresponding equations.

Regarding the curves corresponding to multiplicity sets, the polynomials
$J_{\rho'} \in \ZZ[\rho'][y_1, y_2]$ have degree $\OO(d)$ and bitsize
$\sOO(d^3 + d^2 \tau)$. Therefore, if we substitute $c$ for $y_2$, we end up
with a polynomial in $J_{\rho', c} \in \ZZ[\rho'][y_1]$; it has degree $d$ andbitsize $\sOO(d^8 + d^7\tau)$.

Now, we consider the polynomial $J_{\rho, c} \, J_{\rho', c} \in \ZZ[\rho, \rho'][y_1]$.
Then, among its roots are the $p_1$ and the first coordinate of the closest point of $\cC_{\rho'}$ to  $p'$, say $p'_1$. Therefore, the separation bound $| p_1 - p'_1| \geq 2^{-\sOO(d^{10} + d^9\tau)}$ \cite{st-issac-2012} gives a way to choose (the first coordinate) of the intermediate point $q$.
Hence, the bitsize of the coordinates of $q$ is $\sOO(d^7 + d^6\tau)$.

Consider the system $f_1(x_1, x_2) - y_1 = f_2(x_1, x_2) - y_2 =
\det( \Jac(f_1, f_2)) = 0$. This is a system of three equations in four variables. If we eliminate $x_1$ and $x_2$,
then we obtain a polynomial $\mathcal{D} \in \ZZ[y_1, y_2]$
which is the discriminant of $f$. It has degree $\OO(d^2)$
and its bitsize is $\sOO(d^2 \tau)$.
Then, we work similarly as in the case of the curves that correspond to the multiplicity sets. The bounds for $q$ are similar.

To summarize, if we add a number of order $2^{\sOO(d^{10} + d^9\tau)}$ to the
coordinates of $p$, then we obtain a point $q$ such that the segment $p \, q$
does not intersect neither any other curve corresponding to another multiplicity
set, nor the discriminant curve.

\paragraph{(S2) Count the number of real roots}

Having computed the points $p=(p_1, p_2) \in \cK_{\rho}$ and $q = (q_1, q_2)$
it remains to compute the real roots of the systems $f(x_1, x_2) - p = 0$ and $f(x_1, x_2) - q = 0$.

The second system consists of polynomials in $\ZZ[y_1, y_2]$
of degree $d$ and bitsize $\sOO(d^{10} + d^9\tau)$.

Using \cite{BrSa-0dim-16}, we can solve the system and thus compute its real roots in $\sOB(d^{14.37} + d^{13.37}\tau)$,
using the fastest algorithm for matrix multiplicaiton.

The polynomials of the first system have coefficients in an extension field.
Recall that $p_1 = c \in \ZZ$, where $c \leq 2^{\sOO(d^7 + d^6\tau)}$
  and $p_2$ is a root of a polynomial $J_{\rho}(c, y_1) \in \ZZ[\rho][y_1]$ and $\rho$ is a real root of the polynomial $g(z_1) =0$.
  Thus, we can compute the real roots of $f(y_1, y_2) - p = 0$
  by computing the real solutions of the  system
  \[
    f_1(x_1, x_2) - y_1 =
    f_1(x_1, x_2) - c =
    J_{z_1}(c, y_1) =
    g(z_1) = 0.
  \]
  This is a system of four equations in four unknowns, that is
  $x_1, x_2, y_1, z_1$. The degree of the polynomials is bounded by $\OO(d)$ and
  their bitsize is $\sOO(d^7 + d^6\tau)$. We can solve this system in
  $\sOB(d^{22.11} + d^{21.11}\tau)$ \cite{BrSa-0dim-16}.
  The latter bound dominates the complexity of the overall procedure.

\begin{lemma}
  The complexity of testing emptiness of $\TFRgr$ is
  $\sOB(d^{22.11} + d^{21.11}\tau)$.
\end{lemma}

\begin{theorem}[Complexity of \msresultant over $\RR$]
  \label{thm:cim-res-complexity-R}
  Consider the polynomials $g_1, g_2 \in \ZZ[y_1, y_2][z_1, z_2]$
  which are of degree $d$ with respect to $z_1$ and $z_2$
  and of degree $1$ with respect to $y_1$ and $y_2$.
  Let $g \in \ZZ[z_1]$ be of degree $d$ and bitsize $\tau$.
  The bit complexity of $\msresultant(g_1, g_2, g, \RR)$,
  Alg.~\ref{alg:cim-resultant},
  $\sOB(d^{22.11} + d^{21.11}\tau)$.
\end{theorem}

\section{Implementation}
\label{sec:implementation}

We have implemented in \textsc{maple} a prototype version of our algorithm for
computing the Jelonek set, \sjalg
(Alg.~\ref{alg:sparse-Jelonek-2}). We have also implemented Jelonek's algorithm
\cite{Jel01c} (Alg.~\ref{alg:Jelonek_2}). Our code also uses \textsc{convex}
\cite{convex-soft} to perform some polyhedral computations and to
\textsc{multires}%
\footnote{http://www-sop.inria.fr/members/Bernard.Mourrain/multires.html} to
perform some polynomial manipulations.

A sample use of our software to compute the Jelonek set of the polynomials
$(f_1, f_2) =
1 + x y + 2 x^2 y^2 -7 x^2  y/10 - 3 x^3 y^2,
  g := 1+3 x y -4 x^2 y^2 +5 x^3 y^3  -6 x^4 y^4 + 3^7 x^{10} y^4/2^5  -54 x^{6} y^3 + 5103 x^{9} y^3/320 + x^7 y^2 + x^4 y)
  $
  is as follows:
{\footnotesize
  \begin{lstlisting}
  restart;
  libname := "path to convex library", libname:
  with(convex):
  read("/path to/multires.mpl"):
  read("/path to/Jelonek.mpl"):
  f := 1 + x*y+2*x^2*y^2 -7*x^2 *y/10 - 3*x^3*y^2;
  g := 1+3*x*y-4*x^2*y^2+5*x^3*y^3-6*x^4*y^4+3^7*x^(10)*y^4/2^5
       -54*x^(6)*y^3+5103*x^(9)*y^3/320+x^7*y^2+x^4*y;
  SJ := Sparse_Jelonek_2:
  SJ:-init([f, g]);
  SJ:-compute();
\end{lstlisting}
}
The output is
\[ -150304 + 349920  y_1  = 32 (10935\,y_1-4697) \]

To compute using Jelonek's algorithm,  we type
\begin{lstlisting}
  Jelonek_2([f, g], u, v);
\end{lstlisting}
The output is
  \begin{equation*}
    \begin{aligned}
    -\tfrac{1}{64000}
    ( 4428675\,y_1-4071951 )  ( 729\,y_1-761 )  ( y_1-1 ) ^{2}
    ( 10935\,y_1-4697 ) \\
    ( 9\,{y_1}^{4}-32\,{y_1}^{3}+12\,{y_1}^{2}y_2+5\,{y_1}^{2}+19\,y_1y_2+4\,{y_2}^{2}-35\,y_1-25\,y_2+43 )
    \end{aligned}
\end{equation*}

As we can see, it contains a superset of the Jelonek set.

\vspace{10pt}

\paragraph*{Acknowledgments}
%\subsection*{Acknowledgments}
BEH is supported by the DFG Walter Benjamin Program grant EL 1092/1-1. Part of this work was supported by the Austrian Science Fund (FWF): P33003. This work was initiated while BEH was at the Institute of Mathematics at the Polish Academy
of Sciences in Warsaw. Thanks to Zbigniew Jelonek for presenting the problem and
to Piotr Migus for fruitful discussions.
ET is supported by the ANR JCJC GALOP (ANR-17-CE40-0009), the PGMO grant ALMA,
and the PHC GRAPE.

%%% Local Variables:
%%% mode: latex
%%% TeX-master: "Jelonek_2"
%%% End:

\bibliographystyle{abbrvnat}

\bibliography{proper}

\begin{thebibliography}{46}
\providecommand{\natexlab}[1]{#1}
\providecommand{\url}[1]{\texttt{#1}}
\expandafter\ifx\csname urlstyle\endcsname\relax
  \providecommand{\doi}[1]{doi: #1}\else
  \providecommand{\doi}{doi: \begingroup \urlstyle{rm}\Url}\fi

\bibitem[Basu et~al.(2003)Basu, Pollack, and {M-F.}Roy]{BPR03}
S.~Basu, R.~Pollack, and {M-F.}Roy.
\newblock \emph{Algorithms in Real Algebraic Geometry}, volume~10 of
  \emph{Algorithms and Computation in Mathematics}.
\newblock Springer-Verlag, 2003.
\newblock ISBN 3-540-00973-6.

\bibitem[Bernstein(1975)]{Ber75}
D.~N. Bernstein.
\newblock The number of roots of a system of equations.
\newblock \emph{Funkcional. Anal. i Prilo\v zen.}, 9\penalty0 (3):\penalty0
  1--4, 1975.
\newblock ISSN 0374-1990.

\bibitem[Bivi{\`a}-Ausina(2007)]{Biv07}
C.~Bivi{\`a}-Ausina.
\newblock Injectivity of real polynomial maps and {{\L}}ojasiewicz exponents at
  infinity.
\newblock \emph{Math. Z.}, 257\penalty0 (4):\penalty0 745--767, 2007.

\bibitem[Boissonnat and Teillaud(2007)]{BoiTei-book-2007}
J.-D. Boissonnat and M.~Teillaud.
\newblock \emph{Effective computational geometry for curves and surfaces}.
\newblock Springer, 2007.

\bibitem[Bouzidi et~al.(2016)Bouzidi, Lazard, Moroz, Pouget, Rouillier, and
  Sagraloff]{blmprs-bivsolve-16}
Y.~Bouzidi, S.~Lazard, G.~Moroz, M.~Pouget, F.~Rouillier, and M.~Sagraloff.
\newblock Solving bivariate systems using rational univariate representations.
\newblock \emph{Journal of Complexity}, 37:\penalty0 34--75, 2016.

\bibitem[Brand and Sagraloff(2016)]{BrSa-0dim-16}
C.~Brand and M.~Sagraloff.
\newblock On the complexity of solving zero-dimensional polynomial systems via
  projection.
\newblock In \emph{Proc. ACM on International Symposium on Symbolic and
  Algebraic Computation (ISSAC)}, pages 151--158, 2016.

\bibitem[Chen et~al.(2014)Chen, Dias, Takeuchi, and
  Tib{\u{a}}r]{chen2014invertible}
Y.~Chen, L.~R.~G. Dias, K.~Takeuchi, and M.~Tib{\u{a}}r.
\newblock Invertible polynomial mappings via newton non-degeneracy.
\newblock In \emph{Annales de l'Institut Fourier}, volume~64, pages 1807--1822,
  2014.

\bibitem[Cox et~al.(2006)Cox, Little, and O'shea]{CLO2}
D.~A. Cox, J.~Little, and D.~O'shea.
\newblock \emph{Using algebraic geometry}, volume 185.
\newblock Springer, 2006.

\bibitem[de~Berg et~al.(2008)de~Berg, Cheong, Van~Kreveld, and
  Overmars]{BCKO-CG-08}
M.~de~Berg, O.~Cheong, M.~Van~Kreveld, and M.~Overmars.
\newblock \emph{Computational geometry algorithms and applications}.
\newblock Springer, 3rd edition, 2008.

\bibitem[Diatta et~al.(2022)Diatta, Diatta, Rouillier, Roy, and
  Sagraloff]{DDRRS-curve-18}
D.~N. Diatta, S.~Diatta, F.~Rouillier, M.-F. Roy, and M.~Sagraloff.
\newblock Bounds for polynomials on algebraic numbers and application to curve
  topology.
\newblock \emph{Discrete \& Computational Geometry}, 67\penalty0 (3):\penalty0
  631--697, 2022.

\bibitem[Dickenstein(2016)]{Dic16}
A.~Dickenstein.
\newblock Biochemical reaction networks: an invitation for algebraic geometers.
\newblock In \emph{Mathematical {C}ongress of the {A}mericas}, volume 656 of
  \emph{Contemp. Math.}, pages 65--83. Amer. Math. Soc., Providence, RI, 2016.
\newblock \doi{10.1090/conm/656/13076}.
\newblock URL \url{https://doi.org/10.1090/conm/656/13076}.

\bibitem[Diochnos et~al.(2009)Diochnos, Emiris, and Tsigaridas]{det-jsc-2009}
D.~I. Diochnos, I.~Z. Emiris, and E.~P. Tsigaridas.
\newblock On the asymptotic and practical complexity of solving bivariate
  systems over the reals.
\newblock \emph{Journal of Symbolic Computation}, 44\penalty0 (7):\penalty0
  818--835, 2009.

\bibitem[Duff et~al.(2019)Duff, Kohn, Leykin, and Pajdla]{DKLT-PLMP-19}
T.~Duff, K.~Kohn, A.~Leykin, and T.~Pajdla.
\newblock Plmp-point-line minimal problems in complete multi-view visibility.
\newblock In \emph{Proceedings of the IEEE International Conference on Computer
  Vision}, pages 1675--1684, 2019.

\bibitem[Emiris et~al.(2020)Emiris, Mourrain, and Tsigaridas]{emt-dmm-j}
I.~Emiris, B.~Mourrain, and E.~Tsigaridas.
\newblock Separation bounds for polynomial systems.
\newblock \emph{Journal of Symbolic Computation}, 101:\penalty0 128--151, 2020.
\newblock ISSN 0747-7171.

\bibitem[Esterov(2013)]{Est13}
A.~Esterov.
\newblock The discriminant of a system of equations.
\newblock \emph{Advances in Mathematics}, 245:\penalty0 534--572, 2013.

\bibitem[Fernando and Gamboa(2003)]{Fer03}
J.~F. Fernando and J.~Gamboa.
\newblock Polynomial images of {${\Bbb R}^n$}.
\newblock \emph{J. Pure Appl. Algebra}, 179\penalty0 (3):\penalty0 241--254,
  2003.

\bibitem[Franz(2016)]{convex-soft}
M.~Franz.
\newblock Convex (1.2.0), 2016.
\newblock URL \url{https://math.sci.uwo.ca/~mfranz/convex/}.

\bibitem[Gwo\'zdziewicz(2016)]{Gwozdziewicz2016}
J.~Gwo\'zdziewicz.
\newblock Real jacobian mates.
\newblock \emph{Annales Polonici Mathematici}, 117\penalty0 (3):\penalty0
  207--213, 2016.

\bibitem[Jelonek(1993)]{Jel93}
Z.~Jelonek.
\newblock The set of points at which a polynomial map is not proper.
\newblock In \emph{Annales Polonici Mathematici}, volume~58, pages 259--266.
  Instytut Matematyczny Polskiej Akademii Nauk, 1993.

\bibitem[Jelonek(1999)]{Jel99}
Z.~Jelonek.
\newblock Testing sets for properness of polynomial mappings.
\newblock \emph{Mathematische Annalen}, 315\penalty0 (1):\penalty0 1--35, 1999.

\bibitem[Jelonek(2001)]{Jel01c}
Z.~Jelonek.
\newblock Note about the set {$S_f$} for a polynomial mapping {$f\colon{\mathbb
  C}^2\to{\mathbb C}^2$}.
\newblock \emph{Bull. Polish Acad. Sci. Math.}, 49\penalty0 (1):\penalty0
  67--72, 2001.
\newblock ISSN 0239-7269.

\bibitem[Jelonek(2002)]{Jel02}
Z.~Jelonek.
\newblock Geometry of real polynomial mappings.
\newblock \emph{Mathematische Zeitschrift}, 239\penalty0 (2):\penalty0
  321--333, 2002.

\bibitem[Jelonek and Kurdyka(2003)]{JK03}
Z.~Jelonek and K.~Kurdyka.
\newblock On asymptotic critical values of a complex polynomial.
\newblock \emph{J. Reine Angew. Math.}, 565:\penalty0 1--11, 2003.
\newblock ISSN 0075-4102.
\newblock \doi{10.1515/crll.2003.101}.
\newblock URL \url{https://doi.org/10.1515/crll.2003.101}.

\bibitem[Jelonek and Laso\'{n}(2009)]{JelLas09}
Z.~Jelonek and M.~Laso\'{n}.
\newblock The set of fixed points of a unipotent group.
\newblock \emph{J. Algebra}, 322\penalty0 (6):\penalty0 2180--2185, 2009.
\newblock ISSN 0021-8693.
\newblock \doi{10.1016/j.jalgebra.2009.06.007}.
\newblock URL \url{https://doi.org/10.1016/j.jalgebra.2009.06.007}.

\bibitem[Jelonek and Laso{\'n}(2018)]{JelLas18}
Z.~Jelonek and M.~Laso{\'n}.
\newblock Quantitative properties of the non-properness set of a polynomial
  map.
\newblock \emph{Manuscripta Mathematica}, 156\penalty0 (3-4):\penalty0
  383--397, 2018.

\bibitem[Jelonek and Tib{\u{a}}r(2017)]{JelTib17}
Z.~Jelonek and M.~Tib{\u{a}}r.
\newblock Detecting asymptotic non-regular values by polar curves.
\newblock \emph{International Mathematics Research Notices}, 2017\penalty0
  (3):\penalty0 809--829, 2017.

\bibitem[Katsamaki et~al.(2020)Katsamaki, Rouillier, Tsigaridas, and
  Zafeirakopoulos]{KRTZ-ptopoo-20}
C.~Katsamaki, F.~Rouillier, E.~P. Tsigaridas, and Z.~Zafeirakopoulos.
\newblock On the geometry and the topology of parametric curves.
\newblock In I.~Z. Emiris and L.~Zhi, editors, \emph{{ISSAC} '20: International
  Symposium on Symbolic and Algebraic Computation, Kalamata, Greece, July
  20-23, 2020}, pages 281--288. {ACM}, 2020.

\bibitem[Khovanskii(2016)]{Kho16}
A.~G. Khovanskii.
\newblock Newton polytopes and irreducible components of complete
  intersections.
\newblock \emph{Izvestiya: Mathematics}, 80\penalty0 (1):\penalty0 263, 2016.

\bibitem[Kouchnirenko(1976)]{Kou76}
A.~G. Kouchnirenko.
\newblock Poly{\`e}dres de {N}ewton et nombres de {M}ilnor.
\newblock \emph{Invent. Math.}, 32\penalty0 (1):\penalty0 1--31, 1976.

\bibitem[Langemyr(1990)]{Langemyr-single-90}
L.~Langemyr.
\newblock {An asymptotically fast probabilistic algorithm for computing
  polynomial GCD's over an algebraic number field}.
\newblock In \emph{{International Symposium on Applied Algebra, Algebraic
  Algorithms, and Error-Correcting Codes}}, pages 222--233. Springer, 1990.

\bibitem[Lickteig and Roy(2001)]{LicRoy-sub-res-01}
T.~Lickteig and M.-F. Roy.
\newblock Sylvester{\textendash}{H}abicht sequences and fast {C}auchy index
  computation.
\newblock \emph{J. Symb. Comput.}, 31\penalty0 (3):\penalty0 315--341, Mar.
  2001.
\newblock \doi{10.1006/jsco.2000.0427}.
\newblock URL \url{https://doi.org/10.1006/jsco.2000.0427}.

\bibitem[N{\'e}methi and Zaharia(1990)]{NZ90}
A.~N{\'e}methi and A.~Zaharia.
\newblock On the bifurcation set of a polynomial function and newton boundary.
\newblock \emph{Publications of the Research Institute for Mathematical
  Sciences}, 26\penalty0 (4):\penalty0 681--689, 1990.

\bibitem[Pachter and Sturmfels(2005)]{LS05}
L.~Pachter and B.~Sturmfels, editors.
\newblock \emph{Algebraic statistics for computational biology}.
\newblock Cambridge University Press, New York, 2005.
\newblock ISBN 978-0-521-85700-0; 0-521-85700-7.
\newblock \doi{10.1017/CBO9780511610684}.
\newblock URL \url{https://doi.org/10.1017/CBO9780511610684}.

\bibitem[Pan(2002)]{Pan-usolve-02}
V.~Y. Pan.
\newblock Univariate polynomials: nearly optimal algorithms for numerical
  factorization and root-finding.
\newblock \emph{Journal of Symbolic Computation}, 33\penalty0 (5):\penalty0
  701--733, 2002.

\bibitem[Rouillier(1999)]{Rouillier-rur}
F.~Rouillier.
\newblock Solving zero-dimensional systems through the rational univariate
  representation.
\newblock \emph{Applicable Algebra in Engineering, Communication and
  Computing}, 9\penalty0 (5):\penalty0 433--461, 1999.

\bibitem[Schicho(2020)]{Schicho-rigid-move-20}
J.~Schicho.
\newblock And yet it moves: Paradoxically moving linkages in kinematics.
\newblock \emph{arXiv preprint arXiv:2004.12635}, 2020.

\bibitem[Stasica(2002)]{Sta02}
A.~Stasica.
\newblock An effective description of the {J}elonek set.
\newblock \emph{Journal of Pure and Applied Algebra}, 169\penalty0
  (2-3):\penalty0 321--326, 2002.

\bibitem[Stasica(2005)]{Sta05}
A.~Stasica.
\newblock Geometry of the {J}elonek set.
\newblock \emph{Journal of Pure and Applied Algebra}, 198\penalty0
  (1-3):\penalty0 317--327, 2005.

\bibitem[Storjohann(1996)]{Stor-SNF-96}
A.~Storjohann.
\newblock Near optimal algorithms for computing smith normal forms of integer
  matrices.
\newblock In \emph{Proceedings of the 1996 international symposium on Symbolic
  and algebraic computation}, pages 267--274, 1996.

\bibitem[Strzebonski and Tsigaridas(2019)]{st-sef-jsc-2019}
A.~Strzebonski and E.~Tsigaridas.
\newblock Univariate real root isolation in an extension field and
  applications.
\newblock \emph{Journal of Symbolic Computation}, 92:\penalty0 31--51, 2019.

\bibitem[Strzebo\'nski and Tsigaridas()]{st-issac-2012}
A.~Strzebo\'nski and E.~P. Tsigaridas.
\newblock Univariate real root isolation in multiple extension fields.
\newblock In \emph{Proc. ACM on International Symposium on Symbolic and
  Algebraic Computation (ISSAC)}.

\bibitem[Thao(2009)]{thao2009condition}
N.~T. Thao.
\newblock A condition for the properness of polynomial maps.
\newblock \emph{Vietnam Journal of Mathematics}, 37\penalty0 (1):\penalty0
  113--125, 2009.

\bibitem[Valette-Stasica(2007)]{Sta07}
A.~Valette-Stasica.
\newblock Asymptotic values of polynomial mappings of the real plane.
\newblock \emph{Topology Appl.}, 154\penalty0 (2):\penalty0 443--448, 2007.
\newblock ISSN 0166-8641.
\newblock \doi{10.1016/j.topol.2006.06.001}.
\newblock URL \url{https://doi.org/10.1016/j.topol.2006.06.001}.

\bibitem[Van~den Essen(2012)]{dEss12}
A.~Van~den Essen.
\newblock \emph{Polynomial Automorphisms: and the Jacobian Conjecture}, volume
  190.
\newblock Birkh{\"a}user, 2012.
\newblock \doi{10.1007/978-3-0348-8440-2}.

\bibitem[Yap(2000)]{Yap-algebra-book}
C.-K. Yap.
\newblock \emph{Fundamental problems of algorithmic algebra}.
\newblock Oxford University Press, New York, 2000.
\newblock ISBN 978-0-19-512516-0.

\bibitem[Zaharia(1996)]{Zah96}
A.~Zaharia.
\newblock On the bifurcation set of a polynomial function and {N}ewton
  boundary. {II}.
\newblock \emph{Kodai Math. J.}, 19\penalty0 (2):\penalty0 218--233, 1996.
\newblock ISSN 0386-5991.
\newblock \doi{10.2996/kmj/1138043601}.
\newblock URL \url{https://doi.org/10.2996/kmj/1138043601}.

\end{thebibliography}

\end{document}